\documentclass[11pt]{elsarticle}

\usepackage{graphicx}

\usepackage{epsfig,amsfonts,color}
\usepackage{amsmath}

\newcommand{\cv}[1]{{\color{black} #1}}
\usepackage{amssymb, palatino, geometry,url}
\usepackage{algorithmic}
\usepackage[noresetcount,lined,boxed]{algorithm2e} 

\usepackage[colorlinks=true,linkcolor=blue,citecolor=blue,urlcolor=blue]{hyperref}

\renewcommand{\theequation}{\thesection.\arabic{equation}}

\newcommand{\be}{\begin{equation}}
\newcommand{\ee}{\end{equation}}
\newcommand{\bea}{\begin{eqnarray}}
\newcommand{\eea}{\end{eqnarray}}

\newcommand{\bvec}{\left(\begin{array}{c}}
	\newcommand{\evec}{\end{array}\right)}
\newcommand{\bsub}{\begin{subequations}}
	\newcommand{\esub}{\end{subequations}}

\usepackage{lineno}

\usepackage{tabularx} 
\usepackage{svg} 
\usepackage{makecell} 
\usepackage{amsthm} 
\usepackage{verbatim} 
\usepackage{xcolor}

\providecommand{\e}[1]{\ensuremath{\times 10^{#1}}} 

\newtheorem{theorem}{Theorem} 

\theoremstyle{definition}

\theoremstyle{remark}

\theoremstyle{corollary}

\theoremstyle{lemma}

\begin{document}

\title{Spatio-Temporal Economic Properties of Multi-Product Supply Chains}
	
\author[Madison]{Philip A. Tominac}
\ead{tominac@wisc.edu}

\author[Madison]{Weiqi Zhang}
\ead{wzhang483@wisc.edu}
	
\author[Madison]{Victor M. Zavala\corref{CorrespondingAuthor}}
\ead{victor.zavala@wisc.edu}
	
\address[Madison]{University of Wisconsin-Madison, 1415 Engineering Dr, Madison, WI 53706, USA}
\cortext[CorrespondingAuthor]{Corresponding author: V.~M.~Zavala. Tel. +1 (608) 263-7602. Fax +1 (608) 262-5434}
	
\begin{abstract}
In this work, we analyze the spatio-temporal economic properties of multi-product supply chains. Specifically, we interpret the supply chain as a coordinated market in which stakeholders (suppliers, consumers, and providers of transportation, storage, and transformation services) bid into a market that is cleared by an independent entity to obtain allocations and prices. The proposed model provides a general graph representation of spatio-temporal product transport that helps capture geographical transport, time delays, and storage (temporal transport) in a unified and compact manner. This representation allows us to establish fundamental economic properties for the supply chain (revenue adequacy, cost recovery, and competitiveness) and to establish bounds for space-time prices. To illustrate the concepts, we consider a case study in which organic waste is used for producing biogas and electricity. Our market model shows that incentives for waste storage emerge from electricity demand dynamics and illustrates how space-time price dynamics for waste and derived products emerge from geographical transport and storage.
\end{abstract}
	
\maketitle



\section{Introduction}

Multi-product supply chains (SCs) arise in a wide range of industrial applications such as energy infrastructures (e.g., biofuels from biomass and coupled electrical power-natural gas systems) \citep{Mitridati2020,Duenas2015}, waste infrastructures (e.g., plastic, livestock, food), and chemical manufacturing (e.g., pharma, petrochemical, semiconductors) \citep{Lima2016,Barbosa2014}.  A key defining aspect of the multi-product SCs that we study is the presence of {\em product transformation} (i.e.,  products are processed/combined/transformed to obtain other derived products). Multi-product SCs typically involve a wide range of stakeholders such as suppliers and consumers of products, providers of transport, processing, and storage (inventory) services, and other external actors (e.g., policy). This creates a transaction network that exhibits complex interconnectivity across products, stakeholders, spatial  (geographical) locations, and time (see Figure \ref{FigSCTopology}); moreover, stakeholders are often competitive, strategic, and profit-maximizing entities \citep{Garcia2015,Papageorgiou2009}.  As such, understanding the emergent behavior of product flows and of their inherent monetary values (prices) in multi-product SCs is technically challenging. 

It has been recently observed that SCs exhibit parallels with coordinated markets \citet{Sampat2019waste}; in such markets, stakeholders bid into a coordination system (a coordinator) that determines stakeholder allocations and product prices by solving a clearing problem (a mathematical optimization problem). A coordinated market interpretation of SCs provides a useful view that can help understand their emergent behavior and economic properties. for instance, one can show that an optimal solution of the SC delivers competitive equilibria, cost recovery (no stakeholder incurs financial loss), revenue adequacy (payments collected from consumers balances revenue of service providers), and price boundedness (clearing prices are compatible with bids).  Moreover, a coordinated market setting can offer a potential avenue to manage actual SC operations. For instance, coordinated markets have been successfully deployed in electrical power systems \citep{Blumstein2002,Hogan2002,Bohn1984}; in this context, coordination is key to ensure that consumers have access to a continual supply of electricity (complicated due to capacity limitations of power generators and of the transmission network) and is key to ensure transparent prices that reveal the inherent value of electricity. \citet{Pratt1996} use coordinated markets to handle multiple dairy products (e.g., milk, cheese, yoghurt) over a SC that spans the entire United States. 

Alternative mulit-product SC models and associated economic interpretations have been proposed in the literature. \citet{Mokhlesian2014} developed a bilevel framework for multi-product SCs which determines coordinating prices and inventory in competitive settings. \citet{Saghaeeian2018} present a Stackelberg framework for competitive multi-product markets; these market models aim to capture realistic, strategic behavior but are computationally challenging to handle. Early work in developing coordinated market models for multi-product SCs was presented by \citet{Thomas1996}. Recent work by \citet{Sampat2017,Sampat2019waste}, \citet{Tominac2020}, and \citet{Tominac2021} developed general coordination schemes for handling steady-state, multi-product SCs. This work shows that desirable economic properties of coordinated markets can be obtained under a general setting that involves different types of stakeholders and product transformations. Moreover, large instances of these models can be handled using state-of-the-art optimization solvers.  In this work, we extend this SC modeling setting by capturing dynamic behavior; this allows us to incorporate product storage and transportation delays. We use a graph-theoretic representation that allows us to capture geographical transport, storage, and delays in a unified manner; specifically, we show that these effects can be modeled using space-time flows that transport product cross space and time. This effectively creates a space-time graph representation of the SC, similar in spirit to that used in the network flow literature \citep{ahuja1988network}. We show that the proposed graph-theoretic representation allows us to establish fundamental economic properties for the SC model in a compact and intuitive manner.  

\begin{figure}[!htb]
	\center{\includegraphics[width=100mm]{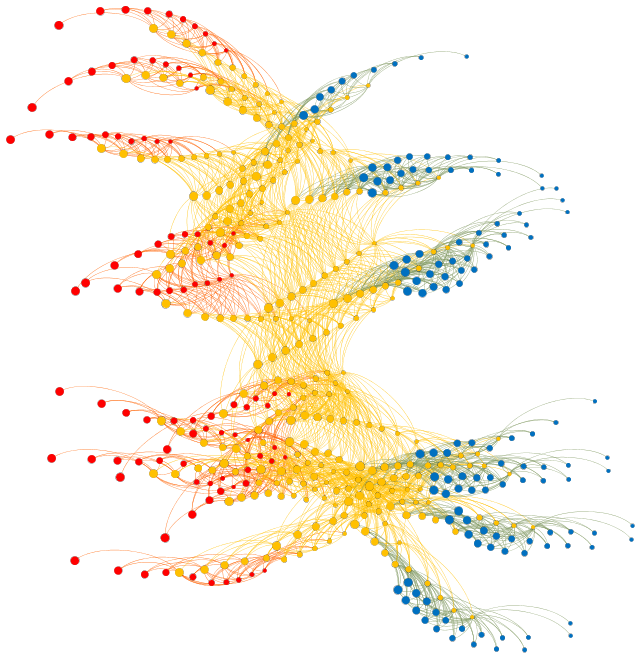}}
	\caption{A supply chain topology can be represented as a graph that connects stakeholders in space and time. Suppliers are visualized in red, technology providers in yellow, and consumers in blue. The arcs represent product transport across spatial and temporal dimensions.}
	\label{FigSCTopology}
\end{figure}

\begin{figure}[!htb]
	\center{\includegraphics[width=150mm]{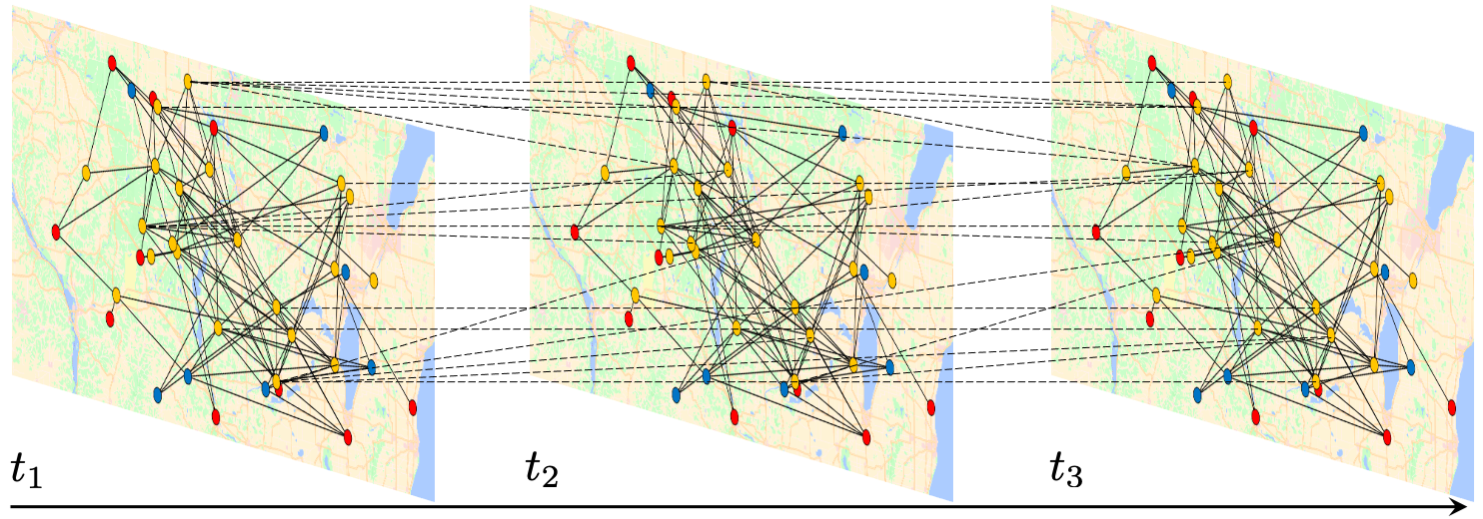}}
	\caption{Geographical visualization of a supply chain. Spatial transport flows (at a given time) are illustrated by solid arrows on the maps (representing individual times) and transport flows traversing time dimension are indicated with dashed lines.}
	\label{FigConcept}
\end{figure}

\section{Supply Chain Model}\label{sec:formula}

The proposed SC model (and associated economic properties) is an extension of the steady-state (SS) model proposed in \citep{Tominac2020} to a dynamical setting. Central to our development is the generalization of the concept of transport product flows from a spatial (geographical) setting to a space-time setting. Space-time flows will allow us to capture a wide range of features (e.g., delays and storage) under a unified framework, will reveal mechanisms that can be used to remunerate stakeholders, and will be used to establish bounding properties for space-time prices. The concept of space-time flows has been recently proposed in the context of electrical power markets (single-product markets) by \citet{zhang2021electricity}; here, we generalize this concept to a multi-product setting. 

To highlight the relevance and properties of space-time flows, we will construct our SC model by first introducing a model under a steady-state setting that only captures spatial transport (we refer to this as the SC-SS model). We will then extend this model by considering a dynamical setting; this model can be seen as a sequence of time snapshots of a steady-state SC that allows for transport of product across time (we refer to this as the SC-ST model). We will see that, by using a graph-theoretical representation, one can represent SC-ST in a form that is directly analogous to SC-SS; this will be key to generalizing economic properties from steady-state to a dynamical setting. The space-time SC setting under study is illustrated in Figure \ref{FigConcept}.

\subsection{Supply Chain Model (Steady-State)}

The steady-state SC setting contains a set of suppliers $\mathcal{G}$, consumers $\mathcal{D}$, geographical (spatial) transport providers $\mathcal{L}$, and product transformation (or technology) providers $\mathcal{M}$. The SC model has an associated graph ${G}(\mathcal{N},\mathcal{A})$, where $\mathcal{N}$ is a set of spatial nodes and $\mathcal{A}$ is the set of undirected arcs (edges, links) connecting such spatial nodes via product transport.  Consumers, suppliers, and technology providers are attached to a given spatial node, while transportation providers are attached to a given arc.  

Each supplier $i\in\mathcal{G}$ has an associated supply flow variable $g_{i}\in\mathbb{R}_{+}$, a supply capacity parameter $\overline{g}_{i}\in\mathbb{R}_{+}$ (representing the maximum flow that it can deliver), and a bid cost parameter $\alpha^{g}_{i}\in\mathbb{R}$. We use $n(i)\in\mathcal{N}$ and $p(i)\in\mathcal{P}$ to denote the node at which supplier is located and the product that it supplies, respectively. If the supply bid cost is positive, the supplier expects payment for product $p(i)$ delivered and, when it is negative, the supplier offers a payment to have $p(i)$ be taken away. This indicates that suppliers may act as either revenue sources or revenue sinks (depending on the sign of their bid); as such, we categorize suppliers using the subsets $\mathcal{G}^{+}\subseteq\mathcal{G}$ of suppliers with $\mathcal{G}^{+}:=\{i\in\mathcal{G}|\alpha^{g}_{i} \geq 0\}$ and $\mathcal{G}^{-}\subseteq\mathcal{G}$ of suppliers with $\alpha^{g}_{i} < 0$. Supplier negative bid costs are common in waste markets (the supplier of waste desires to get rid of it and is willing to pay a tipping fee). We also categorize suppliers by location and product by defining the subsets $\mathcal{G}_{n,p}\subseteq\mathcal{G}$ (where $\mathcal{G}_{n,p}:=\{i|n(i)=n,p(i)=p\}$).

Each consumer $j\in\mathcal{D}$ has a demand flow variable $d_{j}\in\mathbb{R}_{+}$, a demand capacity parameter $\overline{d}_{j}\in\mathbb{R}_{+}$ (indicating the maximum demand flow that it can receive), and a purchase bid cost parameter $\alpha^{d}_{j}\in\mathbb{R}$. We use $n(j)\in\mathcal{N}$ and $p(j)\in\mathcal{P}$ to indicate the consumer location and the product requested. A positive bid cost indicates that a consumer offers payment in exchange for a product while a negative bid indicates that a consumer demands payment on receiving a product. As with suppliers, consumers in a market may act as either revenue sources or sinks. We define the subsets $\mathcal{D}^{+}\subseteq\mathcal{D}$ of consumers with bids $\alpha^{d}_{j} \geq 0$ and $\mathcal{D}^{-}\subseteq\mathcal{D}$ with $\alpha^{d}_{j} < 0$. Consumer negative bid costs are also common in waste markets (a landfill can act as a consumer that requests a tipping fee to take waste). We also categorize consumers by location and product by defining the subsets $\mathcal{D}_{n,p}\subseteq\mathcal{D}$ where $\mathcal{D}_{n,p}:=\{j|n(j)=n,p(j)=p\}$.

Each transport provider $l\in\mathcal{L}$ has a transport flow variable $f_{l}\in\mathbb{R}_{+}$,  capacity parameter $\overline{f}_{l}\in\mathbb{R}_{+}$, and bid cost parameter $\alpha^{f}_{l}\in\mathbb{R}_{+}$ (representing the cost of the transport service). Each transport provider moves product from a base (source) node $n_{b}(l)\in\mathcal{N}$ to a receiving (destination) node $n_{r}(l)\in\mathcal{N}$ via the arc $a:=(n_{b},n_{r})\in\mathcal{A}$ and we use $p(l)\in\mathcal{P}$ to denote the product that is transported. We categorize transport providers using the subsets $\mathcal{L}_{n,p}^{in}\subseteq\mathcal{L}$ and $\mathcal{L}_{n,p}^{out}\subseteq\mathcal{L}$, which correspond to incoming and outgoing providers for a given node $n$ and for a given product $p$; here, $\mathcal{L}_{n,p}^{in}:=\{l|n_{r}(l)=n,p(l)=p\}$  and $\mathcal{L}_{n,p}^{out}:=\{l|n_{b}(l)=n,p(l)=p\}$.  We use the notation $n_b(a)\in \mathcal{N}$ and $n_r(a)\in \mathcal{N}$ to denote the base and receiving nodes of arc $a$, respectively. In graph-theoretical terminology, $n_b(a)$ and $n_r(a)$ are the support nodes of arc $a$.  

Each technology provider $m\in\mathcal{M}$ transforms a set of input products $p\in\mathcal{P}^{con}_{m}\subset\mathcal{P}$ into a set of output products $p\in\mathcal{P}^{gen}_{m}\subset\mathcal{P}$. The provider thus acts simultaneously as a consumer and supplier of products and introduces interconnectivity between  products. Each technology has a flow processing variable $\xi_{m}\in\mathbb{R}_{+}$ representing either consumption or generation flow in terms of a reference product $\bar{p}_{m}\in\mathcal{P}^{con}_{m}$. Product transformations are defined by a set of yield coefficient parameters $\gamma_{m,p}\in\mathbb{R}_{+},p\in\{\mathcal{P}^{con}_{m},\mathcal{P}^{gen}_{m}\}$ with respect to the reference product $\bar{p}_{m}$ such that $\gamma_{m,\bar{p}}=1$. Each technology has an processing capacity $\overline{\xi}_{m}\in\mathbb{R}_{+}$, and a bid cost $\alpha^{\xi}_{m}\in\mathbb{R}_{+}$ (technology service cost) in terms of the reference product $\bar{p}(m)$. We classify these players by the types of products that they consume or generate by using the subsets $\mathcal{M}_{n,p}^{con}\subseteq\mathcal{M}$ and $\mathcal{M}_{n,p}^{gen}\subseteq\mathcal{M}$ where $\mathcal{M}_{n,p}^{con}:=\{m|n(m)=n,p(m)\in\mathcal{P}_{m}^{con}\}$ and $\mathcal{M}_{n,p}^{gen}:=\{m|n(m)=n,p(m)\in\mathcal{P}_{m}^{gen}\}$. Negative bid costs for transport and technology services are not considered, as these do not have practical interpretation; however, the formulation can be easily extended to allow for this.

We interpret the SC model as a market clearing formulation; this interpretation reveals economic properties that can help explain behavior (e.g., product and revenue flows across the system). To establish this connection, we consider that stakeholders submit bidding information that consists of costs $(\alpha^{g},\alpha^{d},\alpha^{f},\alpha^{\xi})$ and capacities $(\overline{g},\overline{d},\overline{f},\overline{\xi})$. An independent entity, called an independent system operator (ISO), clears the market by solving the optimization problem \eqref{PrimalSS} (we call this the SC-SS model). In a market context, this formulation determines product allocations for the stakeholders and determines the inherent economic value for such allocations (prices). The formulation also implicitly provides a mechanism to remunerate stakeholders. 

The objective function \eqref{PrimalObjSS} of the clearing formulation aims to find allocations $(g,d,f,\xi)$ that maximize the demand served and that minimize the service costs. This objective naturally makes allocations to the highest-bidding consumers and the lowest-bidding suppliers, transport providers, and technology providers (prioritizes based on bid costs). The presence of negative bid costs reverses prioritization logic (e.g., the formulation allocates product to suppliers with negative bid costs). Maximizing this objective is equivalent to maximize the total profit of the stakeholders (difference between their revenues and costs); as such, the objective of the clearing formulation is often called the {\em total surplus}.  The dual variables $\pi$ of the nodal product balances \eqref{PrimalClearingSS} (known as {\em clearing constraints}) play a key role in remunerating stakeholders. 

\begin{subequations}\label{PrimalSS}
	\begin{equation}\label{PrimalObjSS}
	\begin{aligned}
	\max_{s,d,f,\xi} \quad \sum_{j \in \mathcal{D}}{\alpha^{d}_{j}}d_{j} -
	\sum_{i \in \mathcal{G}}{\alpha^{g}_{i}}g_{i} -
	\sum_{l\in\mathcal{L}}{\alpha^{f}_{l}f_{l}} - 
	\sum_{m\in\mathcal{M}}{\alpha^{\xi}_{m}\xi_{m}}
	\end{aligned}
	\end{equation}
	\begin{multline}\label{PrimalClearingSS}
	\textrm{s.t.}\;
	\sum_{i\in\mathcal{G}_{n,p}}{g_{i}} +
	\sum_{l\in\mathcal{L}^{in}_{n,p}}{f_{l}} +
	\sum_{m\in\mathcal{M}_{n,p}^{gen}}{\gamma_{m,p}\xi_{m}} =\\
	\sum_{j\in\mathcal{D}_{n,p}}{d_{j}} +
	\sum_{l\in\mathcal{L}^{out}_{n,p}}{f_{l}} +
	\sum_{m\in\mathcal{M}_{n,p}^{con}}{\gamma_{m,p}\xi_{m}},
	\quad (n,t,p)\in\mathcal{N}\times\mathcal{T}\times\mathcal{P},\; (\pi_{n,p})
	\end{multline}
	\begin{equation}\label{PrimalSupSS}
	\begin{aligned}
	g_{i}\leq \overline{g}_{i},\quad i\in\mathcal{G},\; (\overline{\lambda}_{i})
	\end{aligned}
	\end{equation}
	\begin{equation}\label{PrimalDemSS}
	\begin{aligned}
	d_{j}\leq \overline{d}_{j},\quad j\in\mathcal{D},\; (\overline{\lambda}_{j})
	\end{aligned}
	\end{equation}
	\begin{equation}\label{PrimalFloSS}
	\begin{aligned}
	f_{l}\leq \overline{f}_{l},\quad l\in\mathcal{L},\; (\overline{\lambda}_{l})
	\end{aligned}
	\end{equation}
	\begin{equation}\label{PrimalTechSS}
	\begin{aligned}
	\xi_{m}\leq \overline{\xi}_{m},\quad m\in\mathcal{M},\; (\overline{\lambda}_{m})
	\end{aligned}
	\end{equation}
\end{subequations}

\subsection{Supply Chain Model (Dynamical)}

We now extend the steady-state SC model to a dynamical setting by considering a sequence of times $t\in \mathcal{T}$.  The set of times $\mathcal{T}$ is typically known as the time horizon (or planning horizon). In the proposed setting, we interpret a given time $t$ as a {\em node} in a temporal graph;  as such, we construct a space-time graph ${G}(\mathcal{S},\mathcal{A})$ with space-time nodes given by the index pair $s:=(n,t)\in \mathcal{S}$ and with undirected arcs $\mathcal{A}$ connecting space-time nodes. We refer to this model as SC-ST; under this setting, suppliers, consumers, and transformation providers participate at a given spatial node $n\in\mathcal{N}$ and at a given time node $t\in \mathcal{T}$ (or simply $s\in \mathcal{S}$).  In other words, they offer or request product and services at a given space-time node $s$. Moreover, under this setting, transport providers move product across space-time nodes via arcs. The time set $\mathcal{T}$ is ordered (the spatial set $\mathcal{N}$ is not) and it is thus often expressed as $\mathcal{T}=\{t_0,t_1,....,t_T\}$ with $t_0\leq t_1\leq \cdots \leq t_T$. We define the distance between a couple of subsequent time nodes $t_{j}$ and $t_{j+1}$ as $\delta(t_j,t_{j+1})$; this distance is often referred to as the time step.  For simplicity, we assume that all the time steps are equal and given by $\delta(t_j,t_{j+1})=\delta$ for all $j=0,...,T-1$.

Under a space-time setting, each supplier $i\in\mathcal{G}$ has a flow variable $g_{i}\in\mathbb{R}_{+}$, capacity parameter $\overline{g}_{i}\in\mathbb{R}_{+}$, and a bid cost parameter $\alpha^{g}_{i}\in\mathbb{R}$. Moreover, each supplier offers product product $p(i)\in\mathcal{P}$ at a space-time node $s(i)=(n(i),t(i))\in\mathcal{S}$ (with $n(i)\in \mathcal{N}$ and $t(i)\in \mathcal{T}$). As before, we categorize suppliers based on their bid cost as $\mathcal{G}^{+}:=\{i\in\mathcal{G}|\alpha^{g}_{i} \geq 0\}$ and $\mathcal{G}^{-}:=\{i\in\mathcal{G}|\alpha^{g}_{i} < 0\}$ and $\mathcal{G}^{-}\subseteq\mathcal{G}$. We also define subsets $\mathcal{G}_{s,p}\subseteq\mathcal{G}$ (with $\mathcal{G}_{s,p}:=\{i|s(i)=(n,t),p(i)=p\}$) to categorize suppliers by space-time location and product. 

Each consumer $j\in\mathcal{D}$ has a flow variable $d_{j}\in\mathbb{R}_{+}$, capacity parameter $\overline{d}_{j}\in\mathbb{R}_{+}$, and bid cost parameter $\alpha^{d}_{j}\in\mathbb{R}$. As with suppliers, each consumer has an associated space-time node $s(j)=(n(j),t(j))\in\mathcal{S}$ and  product $p(j)\in\mathcal{P}$. We define the subsets $\mathcal{D}^{+}\subseteq\mathcal{D}$ of consumers with bids $\alpha^{d}_{j} \geq 0$ and $\mathcal{D}^{-}\subseteq\mathcal{D}$ with $\alpha^{d}_{j} < 0$. We also define subsets $\mathcal{D}_{s,p}\subseteq\mathcal{D}$ where $\mathcal{D}_{s,p}:=\{j|s(j)=(n,t),p(j)=p\}$.

Each technology provider $m\in\mathcal{M}$ converts a set of input products $p\in\mathcal{P}^{con}_{m}\subset\mathcal{P}$  into a set of output products $p\in\mathcal{P}^{gen}_{m}\subset\mathcal{P}$. Each technology has a variable flow $\xi_{m}\in\mathbb{R}_{+}$ and yield parameters $\gamma_{m,p}\in\mathbb{R}_{+},p\in\{\mathcal{P}^{con}_{m},\mathcal{P}^{gen}_{m}\}$. Each technology has an input capacity defined by $\overline{\xi}_{m}\in\mathbb{R}_{+}$, and a bid  cost $\alpha^{\xi}_{m}\in\mathbb{R}_{+}$. As with suppliers and consumers, each technology provides has an associated space-time node $s(m)\in \mathcal{S}$. We define subsets for technology providers $\mathcal{M}_{n,t,p}^{con}\subseteq\mathcal{M}$ and $\mathcal{M}_{n,t,p}^{gen}\subseteq\mathcal{M}$ where $\mathcal{M}_{s,p}^{con}:=\{m|s(m)=(n,t),p(m)\in\mathcal{P}_{m}^{con}\}$ and $\mathcal{M}_{s,p}^{gen}:=\{m|s(m)=(n,t),p(m)\in\mathcal{P}_{m}^{gen}\}$.

The key difference between the steady-state and the dynamical setting is in how transportation providers are defined and interpreted. Each transport provider $l\in\mathcal{L}$ has a space-time transport flow variable $f_{l}\in\mathbb{R}_{+}$, a capacity parameter $\overline{f}_{l}\in\mathbb{R}_{+}$, and a bid cost parameter $\alpha^{f}_{l}\in\mathbb{R}_{+}$. Each transport provider moves product from a source node $s_{b}(l)\in\mathcal{N}$  to a receiving node $s_{r}(l)\in\mathcal{N}$. The base node is given by $s_{b}(l):=(n_b(l),t_b(l))$ with $n_b(l)\in\mathcal{N}$ and $t_b(l)\in\mathcal{T}$ (a similar definition is used for the receiving node). We define subsets for transport providers $\mathcal{L}_{s,p}^{in}\subseteq\mathcal{L}$ and $\mathcal{L}_{s,p}^{out}\subseteq\mathcal{L}$ corresponding to inbound and outbound transport flow from a space-time node $s$ and for product $p$. Here, $\mathcal{L}_{s,p}^{in}:=\{l|s(l)=s_r(l),p(l)=p\}$ and $\mathcal{L}_{s,p}^{out}:=\{l|s_{b}(l)=s,p(l)=p\}$.  A transport provider moves product across time via the arc $a:=(s_b,s_r)\in\mathcal{A}$;  the arc can also be expressed $a=(n_b,t_b,n_r,t_r)$ to highlight the spatial and temporal nodes. We use the notation $n_b(a)\in \mathcal{N}$ and $t_b(a)\in \mathcal{T}$ to denote the support nodes of the space-time arc.

From the previous definitions it is clear that the concepts of space-time nodes and flows allow us to define transport providers in a manner that is directly analogous to the steady-state counterpart.  We now discuss how the notion of space-time flows offers modeling flexibility to capture different behavior encountered in supply chains. 

\begin{itemize}

\item A transport flow from a spatial location $n_{b}(l)$ to $n_{r}(l)$ at a fixed time location ($t_{b}(l)=t_{r}(l)$) can be used to capture short-term (instantaneous) transport.  This type of spatial transport is the one assumed in the SS-SC model. 

\item A transport flow $l$ from spatial location $n_{b}(l)$ to $n_{r}(l)$  and from time $t_{b}(l)$ to $t_{r}(l)$ (with $t_{r}(l)\geq t_{b}(l)$) can be used to capture long-term transport; here, the time distance $\delta(t_{r}(l),t_{b}(l))$ is the transportation delay.  If $\delta(t_{r}(l),t_{b}(l))=0$, we have the instantaneous case previously described. 

\item A transport flow from time location $t_{b}(l)$ to time location $t_{r}(l)$ (with $t_r(l)\geq t_b(l)$)  at a fixed spatial  location ($n_b(l)=n_r(l)$) can be used to capture product storage (inventory). In other words, product storage can be seen as a form of temporal product transport. 

\end{itemize}

The above discussion reveals that there are different types of arcs present in the SC (spatial, temporal, and spatio-temporal). To capture this categorization, we define the spatial arc set $\mathcal{A}_{\mathcal{N}}:=\{a\in \mathcal{A}\,|\, t_b(a)=t_r(a)\}\subseteq \mathcal{A}$ (sending and receiving time locations are the same), the temporal arc set $\mathcal{A}_{\mathcal{T}}:=\{a\in \mathcal{A}\,|\, n_b(a)=n_r(a)\}$ (sending and receiving spatial locations are the same), and the space-time arc set $\mathcal{A}_{\mathcal{S}}:=\{a\in \mathcal{A}\,|\, n_b(a)\neq n_r(a)\,\textrm{or}\, t_b(a)\neq t_r(a)\}$.  We thus have that the entire arc set is given by $\mathcal{A}=\mathcal{A}_\mathcal{N}\cup \mathcal{A}_\mathcal{T} \cup \mathcal{A}_\mathcal{S}$. 

One can think of the space-time SC as a time sequence of supply chains that are connected in time via temporal and spatiotemporal arcs. This interpretation also suggests that we can partition the set of suppliers $\mathcal{G}$ into the subsets $\mathcal{G}_t:=\{i\in \mathcal{G}\,|\,t(i)=t\}$ and thus $\mathcal{G}=\cup_{t\in \mathcal{T}}\mathcal{G}_t$; one can follow this same reasoning to partition the set of consumers $\mathcal{D}$ and technology providers $\mathcal{M}$. The set of transport providers $\mathcal{L}$ can be partitioned in a form that is analogous to the arc partitioning (in spatial $\mathcal{L}_\mathcal{N}$, temporal $\mathcal{L}_\mathcal{T}$, and spatiotemporal transporters $\mathcal{L}_\mathcal{S}$). In addition, one can also partition the set $\mathcal{L}$ into subsets of the form $\mathcal{L}_t:=\{l\in\mathcal{L}\,|\,t_b(l)=t\}$ and thus $\mathcal{L}=\displaystyle \bigcup_{t\in \mathcal{T}}\mathcal{L}_t$. 

\begin{figure}[!htb]
	\center{\includegraphics[width=100mm]{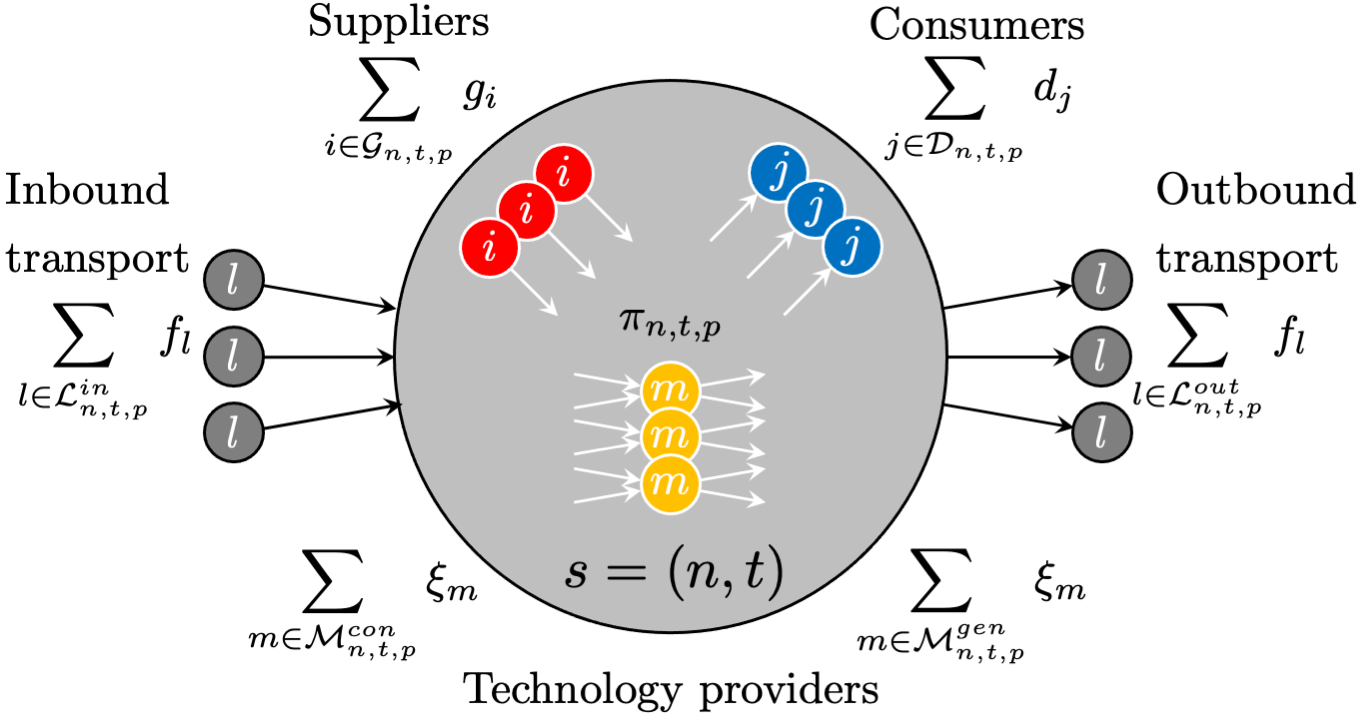}}
	\caption{Schematic of a space-time node $s=(n,t)$ in the SC graph $G(\mathcal{S},\mathcal{A})$. Suppliers ($i\in\mathcal{G}$) consumers ($j\in\mathcal{D}$) and technology providers ($m\in\mathcal{M}$) act on products at a particular node; transportation providers ($l\in\mathcal{L}$) move products across nodes. All the stakeholders influence product prices $\pi_{n,t,p}$ at the space-time node.}
	\label{FigNodeSchematic}
\end{figure}

A special case of the SC-ST model is that in which it is assumed that no temporal and spatiotemporal arcs are present; as such, we have $\mathcal{A}_\mathcal{T}=\emptyset$ and $\mathcal{A}_\mathcal{S}=\emptyset$ and thus $\mathcal{A}=\mathcal{A}_\mathcal{N}$.  We also note that this is equivalent to restricting the flows of the transporters  in the SC-ST model as $f_l=0$ for $l\in \mathcal{A}_{T}\cup \mathcal{A}_{\mathcal{S}}$ (effectively eliminating any arcs that connect nodes across time). In this case, the SC-ST model is simply a time sequence of SC-SS models  that are disconnected in time; this is equivalent to making a quasi steady-state assumption (all transport in time is instantaneous). We thus refer to this special SC-ST model as SC-QSS. The SC-QSS assumption is often made, for instance, if any time delay is much shorter than the time horizon (e.g., time step is in hours while the time horizon spans an entire year). However, our interest in the SC-QSS model arises from the observation that this can be used as a reference model that help us analyze the properties of the SC-ST model. Specifically, we observe that SC-QSS  is a restricted version of SC-ST; this observation will highlight how space-time product transport can be exploited to control space-time price dynamics. 

\subsubsection{SC-ST Formulation (Primal)}

We interpret the SC-ST formulation as a market clearing formulation; this interpretation reveals important and interesting economic properties and behavior of the SC. The formulation is given by: 
\begin{subequations}\label{Primal}
	\begin{equation}\label{PrimalObj}
	\begin{aligned}
	\max_{g,d,f,\xi} \quad \sum_{j \in \mathcal{D}}{\alpha^{d}_{j}}d_{j} -
	\sum_{i \in \mathcal{G}}{\alpha^{g}_{i}}g_{i} -
	\sum_{l\in\mathcal{L}}{\alpha^{f}_{l}f_{l}} - 
	\sum_{m\in\mathcal{M}}{\alpha^{\xi}_{m}\xi_{m}}
	\end{aligned}
	\end{equation}
	\begin{multline}\label{PrimalClearing}
	\textrm{s.t.}\;
	\sum_{i\in\mathcal{G}_{s,p}}{g_{i}} +
	\sum_{l\in\mathcal{L}^{in}_{s,p}}{f_{l}} +
	\sum_{m\in\mathcal{M}_{s,p}^{gen}}{\gamma_{m,p}\xi_{m}} =\\
	\sum_{j\in\mathcal{D}_{s,p}}{d_{j}} +
	\sum_{l\in\mathcal{L}^{out}_{s,p}}{f_{l}} +
	\sum_{m\in\mathcal{M}_{s,p}^{con}}{\gamma_{m,p}\xi_{m}},
	\quad (s,p)\in\mathcal{S}\times\mathcal{P},\; (\pi_{s,p})
	\end{multline}
	\begin{equation}\label{PrimalSup}
	\begin{aligned}
	g_{i}\leq \overline{g}_{i},\quad i\in\mathcal{G},\; (\overline{\lambda}_{i})
	\end{aligned}
	\end{equation}
	\begin{equation}\label{PrimalDem}
	\begin{aligned}
	d_{j}\leq \overline{d}_{j},\quad j\in\mathcal{D},\; (\overline{\lambda}_{j})
	\end{aligned}
	\end{equation}
	\begin{equation}\label{PrimalFlo}
	\begin{aligned}
	f_{l}\leq \overline{f}_{l},\quad l\in\mathcal{L},\; (\overline{\lambda}_{l})
	\end{aligned}
	\end{equation}
	\begin{equation}\label{PrimalTech}
	\begin{aligned}
	\xi_{m}\leq \overline{\xi}_{m},\quad m\in\mathcal{M},\; (\overline{\lambda}_{m})
	\end{aligned}
	\end{equation}
\end{subequations}
We note that this formulation is directly analogous to that of SC-SS formulation \eqref{PrimalSS}; the key difference is that SC-SS uses the spatial node set $\mathcal{N}$, while the SC-ST formulation \eqref{Primal} uses the space-time node set $\mathcal{S}$.  This observation is key to establish economic properties for the SC under a dynamic setting. In the SC-ST formulation, the market clearing constraints capture product balances at different spatial locations and at different times. Here, it is clear that transport of product across space and time is driven by the transport flows. The dual variable of the clearing constraint is denoted $\pi_{s,p}\in\mathbb{R}$;  the optimal value for this variable is the clearing price of product $p$ at space-time node $s=(n,t)$. This reveals that clearing prices exhibit space-time behavior. 

Equations \eqref{PrimalSup} to \eqref{PrimalTech} enforce capacity constraints for each stakeholder. The corresponding dual variables $\overline{\lambda}_{i}\in\mathbb{R}_{+}$, $\overline{\lambda}_{j}\in\mathbb{R}_{+}$, $\overline{\lambda}_{l}\in\mathbb{R}_{+}$, and $\overline{\lambda}_{m}\in\mathbb{R}_{+}$ are highlighted here. \citet{Tominac2020} demonstrate that non-zero lower bounds create artificial incentives that interfere with market clearing prices; as such, we do not allow for non-zero lower bounds on allocations.

If we express the space-time nodes in disaggregated form $s=(n,t)$ and we use set time partitions, we can express the SC-ST model in the equivalent form:

\begin{subequations}\label{PrimalNT}
	\begin{equation}\label{PrimalObjNT}
	\begin{aligned}
	\max_{g,d,f,\xi} \quad \sum_{t\in\mathcal{T}}\left(\sum_{j \in \mathcal{D}_t}{\alpha^{d}_{j}}d_{j} -
	\sum_{i \in \mathcal{G}_t}{\alpha^{g}_{i}}g_{i} -
	\sum_{l\in\mathcal{L}_t}{\alpha^{f}_{l}f_{l}} - 
	\sum_{m\in\mathcal{M}_t}{\alpha^{\xi}_{m}\xi_{m}}\right)
	\end{aligned}
	\end{equation}
	\begin{multline}\label{PrimalClearingNT}
	\textrm{s.t.}\;
	\sum_{i\in\mathcal{G}_{n,t,p}}{g_{i}} +
	\sum_{l\in\mathcal{L}^{in}_{n,t,p}}{f_{l}} +
	\sum_{m\in\mathcal{M}_{n,t,p}^{gen}}{\gamma_{m,p}\xi_{m}} =\\
	\sum_{j\in\mathcal{D}_{n,t,p}}{d_{j}} +
	\sum_{l\in\mathcal{L}^{out}_{n,t,p}}{f_{l}} +
	\sum_{m\in\mathcal{M}_{n,t,p}^{con}}{\gamma_{m,p}\xi_{m}},
	\quad (n,t,p)\in\mathcal{N}\times\mathcal{T}\times\mathcal{P},\; (\pi_{n,t,p})
	\end{multline}
	\begin{equation}\label{PrimalSupNT}
	\begin{aligned}
	g_{i}\leq \overline{g}_{i},\quad i\in\mathcal{G},\; (\overline{\lambda}_{i})
	\end{aligned}
	\end{equation}
	\begin{equation}\label{PrimalDemNT}
	\begin{aligned}
	d_{j}\leq \overline{d}_{j},\quad j\in\mathcal{D},\; (\overline{\lambda}_{j})
	\end{aligned}
	\end{equation}
	\begin{equation}\label{PrimalFloNT}
	\begin{aligned}
	f_{l}\leq \overline{f}_{l},\quad l\in\mathcal{L},\; (\overline{\lambda}_{l})
	\end{aligned}
	\end{equation}
	\begin{equation}\label{PrimalTechNT}
	\begin{aligned}
	\xi_{m}\leq \overline{\xi}_{m},\quad m\in\mathcal{M},\; (\overline{\lambda}_{m})
	\end{aligned}
	\end{equation}
\end{subequations}
This formulation is more verbose but reveals the space-time nature of the problem. Specifically, this reveals that the dynamic SC is indeed a time sequence of SS-SCs that are interconnected via temporal transport flows. Moreover, the total surplus is the summation of the surplus at the different times. 

\subsubsection{SC-ST Formulation (Dual)}\label{SecDual}

The SS-ST formulation can be expressed in the dual form:

\begin{subequations}\label{Dual}
	\begin{equation}\label{DualObj}
	\begin{aligned}
	\min_{\pi,\overline{\lambda}}\; 
	\sum_{i \in \mathcal{G}}{\overline{g}_{i}\overline{\lambda}_{i}} +
	\sum_{j \in \mathcal{D}}{\overline{d}_{j}\overline{\lambda}_{j}} +
	\sum_{l\in\mathcal{L}}{\overline{f}_{l}\overline{\lambda}_{l}} +
	\sum_{m\in\mathcal{M}}{\overline{\xi}_{m}\overline{\lambda}_{m}}
	\end{aligned}
	\end{equation}
	\begin{equation}\label{DualSup}
	\begin{aligned}
	\textrm{s.t.}\;\; \pi_{n(i),t(i),p(i)} - \overline{\lambda}_{i} \leq  \alpha^{g}_{i},\; i\in\mathcal{G}
	\end{aligned}
	\end{equation}
	\begin{equation}\label{DualDem}
	\begin{aligned}
	\pi_{n(j),t(j),p(j)} + \overline{\lambda}_{j} \geq \alpha^{d}_{j},\; j\in\mathcal{D}
	\end{aligned}
	\end{equation}
	\begin{equation}\label{DualFlo}
	\begin{aligned}
	\pi_{n_{r}(l),t_{r}(l),p(l)} - \pi_{n_{b}(l),t_{b}(l),p(l)} - \overline{\lambda}_{l} \leq \alpha^{f}_{l},\; l\in\mathcal{L}
	\end{aligned}
	\end{equation}
	\begin{equation}\label{DualTech}
	\begin{aligned}
	\sum_{p\in\mathcal{P}^{gen}_{m}}{\gamma_{m,p}\pi_{n(m),t(m),p(m)}} - \sum_{p\in\mathcal{P}^{con}_{m}}{\gamma_{m,p}\pi_{n(m),t(m),p(m)}} - \overline{\lambda}_{m} \leq \alpha^{\xi}_{m},\; m\in\mathcal{M}
	\end{aligned}
	\end{equation}
\end{subequations}

This formulation more clearly reveals the relationship between the prices $\pi$ and the capacity dual variables $\overline{\lambda}$ (which can be interpreted as marginal profits). The dual problem also immediately reveals bounding properties for clearing prices. The dual shown in \eqref{DualObj} does not depend on market prices; this only depends on the marginal profit variables $\overline{\lambda}$. Viewed through complementary slackness, the dual formulation minimizes the combined value of stakeholder marginal profits.

Constraints \eqref{DualSup} to \eqref{DualTech} correspond to the primal allocations $(g,d,f,\xi)$ and define spatiotemporal prices associated with each stakeholder class in terms of the clearing prices $\pi_{n,t,p}$. We define the price identities $(\pi_{i},\pi_{j},\pi_{l},\pi_{m})$ in \eqref{PriceIdentities}. The price identities \eqref{PI1} to \eqref{PI5} are interpreted as the supply price, demand price, transportation price, and technology (or transformation) price. These identities demonstrate that supply and demand prices are equivalent to the corresponding nodal product prices; this means that suppliers and consumers pay (or are paid) according to the product value at their space-time location. The transport price is the difference between prices at different space-time nodes. The technology price is a yield-weighted difference between nodal prices of the outputs of a given technology and its inputs. A solution of the dual problem provides these prices, and substitution of these identities into \eqref{Dual} condenses the dual to the form shown in \eqref{DualSub}.

\begin{subequations}\label{PriceIdentities}
	\begin{equation}\label{PI1}
	\begin{aligned}
	\pi_{i} := \pi_{n(i),t(i),p(i)},\; i\in\mathcal{G}
	\end{aligned}
	\end{equation}
	\begin{equation}
	\begin{aligned}
	\pi_{j} := \pi_{n(j),t(j),p(j)},\; j\in\mathcal{D}
	\end{aligned}
	\end{equation}
	\begin{equation}\label{TransportPrice}
	\begin{aligned}
	\pi_{l} := \pi_{n_{r}(l),t_{r}(l),p(l)} - \pi_{n_{b}(l),t_{b}(l),p(l)},\; l\in\mathcal{L}
	\end{aligned}
	\end{equation}
	\begin{equation}\label{PI5}
	\begin{aligned}
	\pi_{m} := \sum_{p\in\mathcal{P}^{gen}_{m}}{\gamma_{m,p}\pi_{n(m),t(m),p(m)}} - \sum_{p\in\mathcal{P}^{con}_{m}}{\gamma_{m,p}\pi_{n(m),t(m),p(m)}},\; m\in\mathcal{M}
	\end{aligned}
	\end{equation}
\end{subequations}

We note that the price identities in \eqref{PriceIdentities} imply additional relationships between the nodal prices and the stakeholder prices, as shown in \eqref{PriceIdentitiesPt2}.

\begin{subequations}\label{PriceIdentitiesPt2}
	\begin{equation}
	\begin{aligned}
	\sum_{n\in\mathcal{N}}\sum_{t\in\mathcal{T}}\sum_{p\in\mathcal{P}}\pi_{n,t,p}\sum_{i\in\mathcal{G}_{n,t,p}}g_{i} = \sum_{i\in\mathcal{G}}\pi_{i}g_{i}
	\end{aligned}
	\end{equation}
	\begin{equation}
	\begin{aligned}
	\sum_{n\in\mathcal{N}}\sum_{t\in\mathcal{T}}\sum_{p\in\mathcal{P}}\pi_{n,t,p}\sum_{j\in\mathcal{D}_{n,t,p}}d_{j} = \sum_{j\in\mathcal{D}}\pi_{j}d_{j}
	\end{aligned}
	\end{equation}
	\begin{equation}
	\begin{aligned}
	\sum_{n\in\mathcal{N}}\sum_{t\in\mathcal{T}}\sum_{p\in\mathcal{P}}\pi_{n,t,p}\left(\sum_{l\in\mathcal{L}^{in}_{n,t,p}}f_{l} - \sum_{l\in\mathcal{L}^{out}_{n,t,p}}f_{l}\right) = \sum_{l\in\mathcal{L}}\pi_{l}f_{l}
	\end{aligned}
	\end{equation}
	\begin{equation}
	\begin{aligned}
	\sum_{n\in\mathcal{N}}\sum_{t\in\mathcal{T}}\sum_{p\in\mathcal{P}}\pi_{n,t,p}\left(\sum_{m\in\mathcal{M}^{gen}_{n,t,p}}\gamma_{m,p}\xi_{m} - \sum_{m\in\mathcal{M}^{con}_{n,t,p}}\gamma_{m,p}\xi_{m}\right) = \sum_{m\in\mathcal{M}}\pi_{m}\xi_{m}
	\end{aligned}
	\end{equation}
\end{subequations}

Substitution of the identities into the dual program yields \eqref{DualSub}; this formulation illustrates price bounding relationships governing the coordination behavior. Dual constraint \eqref{DualSubDem} provides a lower bound on consumer prices $\pi_{j}$, while the remaining dual constraints provide upper bounds on the stakeholder prices $\pi_{i}$, $\pi_{l}$, and $\pi_{m}$ for suppliers, transport providers, and technology providers, respectively. These bounds are important in understanding price behavior (e.g., space-time dynamics) that arise under coordination.

\begin{subequations}\label{DualSub}
	\begin{equation}
	\begin{aligned}
	\min_{\pi,\overline{\lambda}}
	\sum_{i \in \mathcal{G}}{\overline{g}_{i}\overline{\lambda}_{i}} +
	\sum_{j \in \mathcal{D}}{\overline{d}_{j}\overline{\lambda}_{j}} +
	\sum_{l\in\mathcal{L}}{\overline{f}_{l}\overline{\lambda}_{l}} +
	\sum_{m\in\mathcal{M}}{\overline{\xi}_{m}\overline{\lambda}_{m}}
	\end{aligned}
	\end{equation}
	\begin{equation}\label{DualSubSup}
	\begin{aligned}
	\textrm{s.t.}\; \pi_{i} - \overline{\lambda}_{i} \leq  \alpha^{g}_{i},\; i\in\mathcal{G}
	\end{aligned}
	\end{equation}
	\begin{equation}\label{DualSubDem}
	\begin{aligned}
	\pi_{j} + \overline{\lambda}_{j} \geq \alpha^{d}_{j},\; j\in\mathcal{D}
	\end{aligned}
	\end{equation}
	\begin{equation}\label{DualSubFlo}
	\begin{aligned}
	\pi_{l} - \overline{\lambda}_{l} \leq \alpha^{f}_{l},\; l\in\mathcal{L}
	\end{aligned}
	\end{equation}
	\begin{equation}\label{DualSubTech}
	\begin{aligned}
	\pi_{m} - \overline{\lambda}_{m} \leq \alpha^{\xi}_{m},\; m\in\mathcal{M}
	\end{aligned}
	\end{equation}
\end{subequations}

The primal and dual program relationships with individual stakeholder objectives are reproduced in Figure~\ref{FigPrimalDual}, which illustrates how bidding information flows from individual stakeholders to the coordinator. The coordinator determines product allocations and market prices. We will see that market coordination aims to maximize the individual profits of all stakeholders and yields a competitive equilibrium.

\begin{figure}[!htb]
	\center{\includegraphics[width=0.9\textwidth]{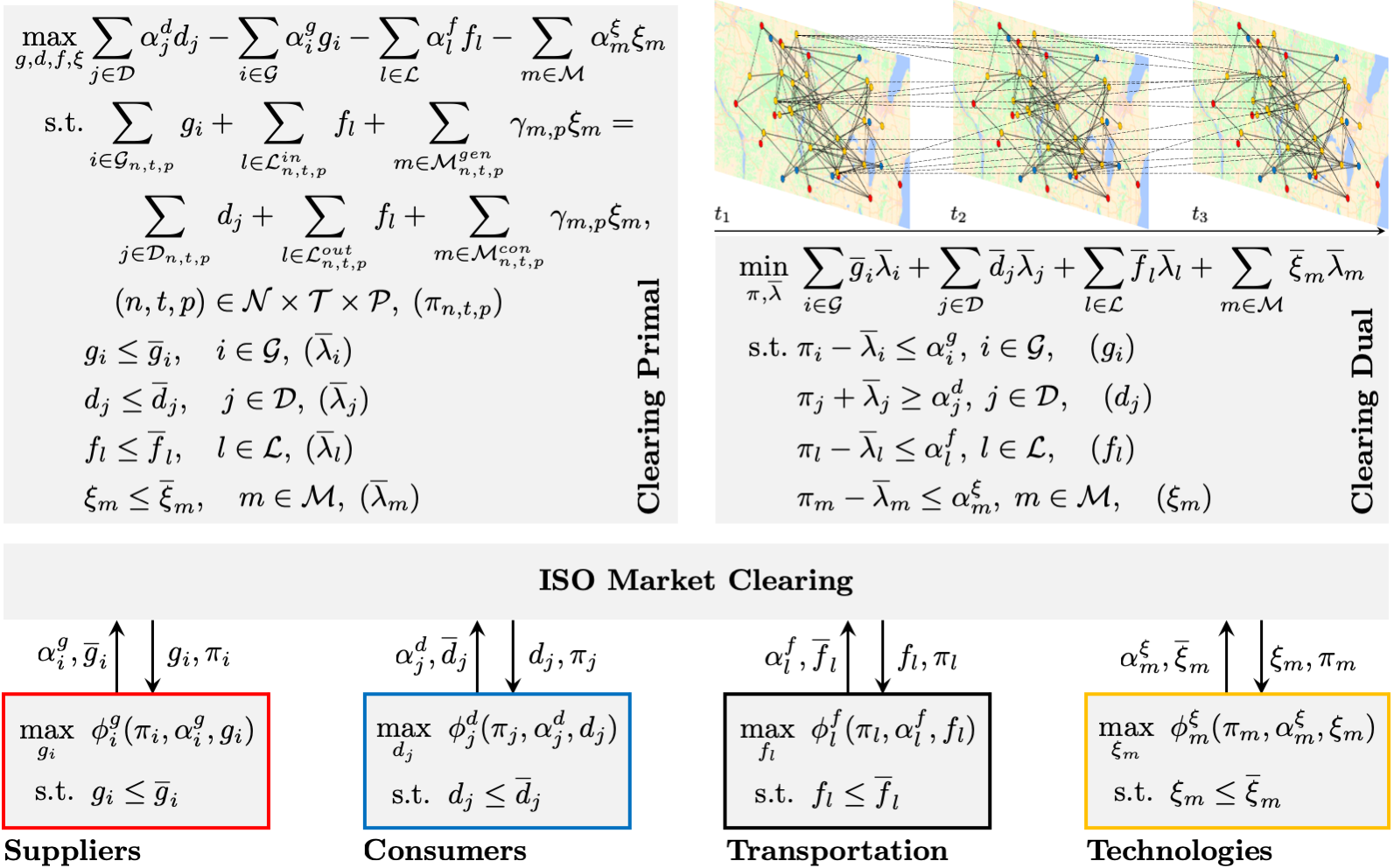}}
	\caption{Primal and dual view of the SC-ST (market clearing) model. The coordinator aims to maximize all the player profits; maximizing each profit subject to the bidding information it receives. The coordinator determines the allocations and market prices that solve the clearing problem.}
	\label{FigPrimalDual}
\end{figure}

\subsubsection{SC-ST Formulation (Lagrangian Dual)}

The Lagrangian dual formulation of SC-ST is:
\begin{equation}\label{LDualRearranged}
\begin{aligned}
\max_{\pi,\overline{\lambda}}\min_{g,d,f,\xi}\mathcal{L}(g,d,f,\xi,\pi,\overline{\lambda})=&
\sum_{i\in\mathcal{G}}-(\pi_{i} -  \alpha^{g}_{i})g_{i} + (g_{i}-\overline{g}_{i})\overline{\lambda}_{i} +\\&
\sum_{j\in\mathcal{D}}-(\alpha^{d}_{j} - \pi_{j})d_{j} + (d_{j}-\overline{d}_{j})\overline{\lambda}_{j} +\\&
\sum_{l\in\mathcal{L}}-(\pi_{l} - \alpha^{f}_{l})f_{l} + (f_{l} - \overline{f}_{l})\overline{\lambda}_{l} +\\&
\sum_{m\in\mathcal{M}}-(\pi_{m} - \alpha^{\xi}_{m})\xi_{m} + (\xi_{m} - \overline{\xi}_{m})\overline{\lambda}_{m}
\end{aligned}
\end{equation}
Here, we denote the Lagrangian function as $\mathcal{L}(g,d,f,\xi,\pi,\overline{\lambda})$ and the Lagrangian dual problem is stated in \eqref{LDualRearranged}. This formulation will reveal that objective function of the clearing problem aims to maximize the stakeholder profits and provides additional price bounding information.  Specifically, we observe that all of the terms of the form $(x_{u} - \overline{x}_{u})\overline{\lambda}_{u}$ are identically zero (either $x_{u} - \overline{x}_{u}=0$ or $\overline{\lambda}_{u}=0$ by complementary slackness). We express the remaining terms to obtain: 
\begin{equation}\label{LDualReduced}
\begin{aligned}
\max_{\pi,\overline{\lambda}}\min_{g,d,f,\xi}\mathcal{L}(g,d,v,f,\xi,\pi,\overline{\lambda})=&
\sum_{i\in\mathcal{g}}-(\pi_{i} -  \alpha^{g}_{i})g_{i} +
\sum_{j\in\mathcal{D}}-(\alpha^{d}_{j} - \pi_{j})d_{j} +\\&
\sum_{l\in\mathcal{L}}-(\pi_{l} - \alpha^{f}_{l})f_{l} +
\sum_{m\in\mathcal{M}}-(\pi_{m} - \alpha^{\xi}_{m})\xi_{m}
\end{aligned}
\end{equation}
We define the profit allocated to a stakeholder by the market as the difference between its revenue and costs. We assume that stakeholders bid their marginal value for a product (their operating costs); this assumption is consistent with the bidding outcomes in a Vickrey-Clarke-Groves auction, in which bidders are incentivized to bid their marginal value \citep{Vickrey1961,Clarke1971,Groves1973}. We denote the vector of stakeholder profits $\phi=(\phi^{g},\phi^{d},\phi^{f},\phi^{\xi})$, where individual stakeholder profits are:
\begin{subequations}\label{ProfitIdentities}
	\begin{equation}
	\begin{aligned}
	\phi^{g}_{i} := (\pi_{i} -  \alpha^{g}_{i})g_{i},\; i\in\mathcal{G}
	\end{aligned}
	\end{equation}
	\begin{equation}
	\begin{aligned}
	\phi^{d}_{j} := (\alpha^{d}_{j} - \pi_{j})d_{j},\; j\in\mathcal{D}
	\end{aligned}
	\end{equation}
	\begin{equation}
	\begin{aligned}
	\phi^{f}_{l} := (\pi_{l} -  \alpha^{f}_{l})f_{l},\; l\in\mathcal{L}
	\end{aligned}
	\end{equation}
	\begin{equation}
	\begin{aligned}
	\phi^{\xi}_{m} := (\pi_{m} -  \alpha^{\xi}_{m})\xi_{m},\; m\in\mathcal{M}
	\end{aligned}
	\end{equation}
\end{subequations}
We note that consumer profits are defined differently from those of service providers. Consumer profit is to be interpreted as money saved; the difference between the consumer bid (consumer willingness to pay) and the market price (that is paid) for the consumer allocation. This difference in the formulation emerges from the expectation that consumers will be revenue sources and service providers will be revenue sinks. These identities are logically consistent for those with negative bids (i.e., $i\in\mathcal{G}^{-}$ and $j\in\mathcal{D}^{-}$).  

\begin{equation}\label{LDual}
\begin{aligned}
\max_{\pi,\overline{\lambda}}\min_{g,d,f,\xi}\mathcal{L}(g,d,f,\xi,\pi,\overline{\lambda})=
-\left(\sum_{i\in\mathcal{G}}\phi^{g}_{i}+
\sum_{j\in\mathcal{D}}\phi^{d}_{j}+
\sum_{l\in\mathcal{L}}\phi^{f}_{l}+
\sum_{m\in\mathcal{M}}\phi^{\xi}_{m}\right)
\end{aligned}
\end{equation}
Substitution of the profit identities into \eqref{LDualReduced} results in \eqref{LDual}, from which we conclude that the Lagrangian dual maximizes the sum of the stakeholder profits.

\section{Economic Properties of SC-ST Model}\label{sec:theorems}

In this section we establish economic properties of the SC; these properties leverage the clearing interpretation of the SC-ST formulation. We begin by establishing that coordination maximizes stakeholder profits and that the stakeholder profits are all nonnegative (regardless of the market outcome).

\begin{theorem}\label{ThmProfit}
	The SC-ST formulation delivers prices $\pi$ and allocations $(g,d,f,\xi)$ that maximize the collective stakeholder profit; moreover, the profits are all nonnegative.
\end{theorem}
\begin{proof}

For an arbitrary set of prices $\pi$ (and the associated $\overline{\lambda}$) the allocations $(g,d,f,\xi)=(0,0,0,0)$ result in a value of the Lagrangian function \eqref{LDual} of zero, i.e., $\mathcal{L}(0,0,0,0,\pi,\overline{\lambda})=0$. Solving the Lagrangian problem produces an allocation $(g^{*},d^{*},f^{*},\xi^{*})$ minimizing the Lagrangian with
\begin{align*}
\mathcal{L}(g^{*},d^{*},f^{*},\xi^{*},\pi,\overline{\lambda})\leq 0. 
\end{align*}
Under fixed prices, the Lagrangian is the sum of player profits and thus the allocation $(g^{*},d^{*},f^{*},\xi^{*})$ results in profits that are no worse than $(0,0,0,0)$. It thus follows that the profits $(\phi^{g},\phi^{d},\phi^{f},\phi^{\xi})$ are nonnegative. Since this is true for arbitrary prices, then it holds for the optimal prices $\pi^{*}$, and that profits are nonnegative for optimal allocations and prices.
\end{proof}

Theorem \ref{ThmProfit} provides the groundwork necessary to establish that a solution of the SC-ST formulation is better than the solution of the SC-QSS problem (in terms of total surplus). Here, we define the optimal total surplus of SC-ST as $\varphi^*$ and the optimal total surplus of SC-QSS as $\hat{\varphi}$. 

\begin{theorem}\label{ThmSocialWelfare}
	The optimal total surplus of SC-ST $\varphi^*$  and of SC-QSS $\hat{\varphi}$ satisfy $\varphi^*\geq \hat{\varphi}$. 
\end{theorem}

\begin{proof}
The SC-QSS problem can be obtained from SC-ST by imposing the flow constraints $f_l=0$ for $l\in \mathcal{A}_{T}\cup \mathcal{A}_{\mathcal{S}}$. As such, the feasible region of SC-QSS is contained in the that of SC-ST. The result follows. 
\end{proof}

This result is important because it highlights that temporal transport of product can {\em add flexibility} to improve the total surplus.  We next establish that the solution of SC-ST gives a competitive equilibrium. 

\begin{theorem}\label{ThmCompetition}
	A solution of SC-ST delivers prices $\pi$ and allocations $(g,d,f,\xi)$ that represent a competitive equilibrium.
\end{theorem}
\begin{proof}
	It is sufficient to show that the solution of SC-ST maximizes profits subject to the clearing constraints. For arbitrary prices $\pi$ solving the Lagrangian \eqref{LDual} produces allocations $(g^{*},d^{*},f^{*},\xi^{*})$ maximizing stakeholder profits (independently). The market clearing program is linear, and strong duality implies that the allocations $(g^{*},d^{*},f^{*},\xi^{*})$ also satisfy the clearing constraints in the primal program \eqref{Primal}. 
\end{proof}

Coordination should yield consistent properties with respect to product prices and revenue streams. Importantly, SC-ST must satisfies a condition termed \textit{revenue adequacy}, which states that revenue streams collected from revenue sources is sufficient to remunerate revenue sinks.
\\

\begin{theorem}\label{ThmPrices}
	A solution of SC-ST delivers clearing prices $\pi$ and allocations $(g,d,f,\xi)$ that satisfy revenue adequacy: 
	\begin{equation}\label{RevenueAdequacy}
	\begin{aligned}
	\sum_{j\in\mathcal{D}^{+}}\pi_{j}d_{j} + \sum_{i\in\mathcal{G}^{-}}\pi_{i}g_{i} =
	\sum_{j\in\mathcal{D}^{-}}\pi_{j}d_{j} + \sum_{i\in\mathcal{G}^{+}}\pi_{i}g_{i} + \sum_{l\in\mathcal{L}}\pi_{l}f_{l} + \sum_{m\in\mathcal{M}}\pi_{m}\xi_{m}
	\end{aligned}
	\end{equation}
\end{theorem}
\begin{proof}
	Consider the market clearing constraint \eqref{PrimalClearing}. We obtain all spatiotemporal revenue streams by multiplying the product allocations at each node and time point by their corresponding dual prices $\pi_{n,t,p}$, and adding over the node, time, and product dimensions in \eqref{RA1}
	
	\begin{equation}\label{RA1}
	\begin{aligned}
	\sum_{n\in\mathcal{N}}\sum_{t\in\mathcal{T}}\sum_{p\in\mathcal{P}}\pi_{n,t,p}
	\left(
	\sum_{i\in\mathcal{G}_{n,t,p}}{g_{i}} +
	\sum_{l\in\mathcal{L}^{in}_{n,t,p}}{f_{l}} +
	\sum_{m\in\mathcal{M}_{n,t,p}^{gen}}{\gamma_{m,p}\xi_{m}} -\right.\\\left.
	\sum_{j\in\mathcal{D}_{n,t,p}}{d_{j}} -
	\sum_{l\in\mathcal{L}^{out}_{n,t,p}}{f_{l}} -
	\sum_{m\in\mathcal{M}_{n,t,p}^{con}}{\gamma_{m,p}\xi_{m}}
	\right)
	\end{aligned}
	\end{equation}
	
	Applying the price identities in \eqref{PriceIdentitiesPt2} results in: 
	\begin{equation}\label{RA2}
	\begin{aligned}
	 \sum_{i\in\mathcal{G}}\pi_{i}g_{i} + \sum_{l\in\mathcal{L}}\pi_{l}f_{l} + \sum_{m\in\mathcal{M}}\pi_{m}\xi_{m} - \sum_{j\in\mathcal{D}}\pi_{j}d_{j}=0
	\end{aligned}
	\end{equation}
	This is identical to \eqref{RevenueAdequacy} (with suppliers and consumers grouped by their bid signs).
\end{proof}

An important implication of revenue adequacy is that coordination does not introduce inefficiencies into the market. Moreover, revenue adequacy implies that revenue streams may flow both forward and backward through time. One way of interpreting this result is that the promise of future payment creates incentives to move products to future time periods. Revenue adequacy is sometimes called \textit{cost recovery} in electricity markets, and provides electricity buyers and sellers with theoretical guarantees of price behavior and confidence in the competitiveness of market outcomes.  The derived revenue adequacy result provides a compact and intuitive view on how economic value (revenue) is preserved in space-time (economic value is conserved). 

\begin{theorem}\label{ThmLowerBounds}
	The clearing prices $(\pi_{i},\pi_{j},\pi_{l},\pi_{m})$ corresponding to the cleared players $(\mathcal{G}^{*},\mathcal{D}^{*},\mathcal{L}^{*},\mathcal{M}^{*})$ satisfy the bounds $\pi_{i}\geq\alpha^{g}_{i},\;i\in\mathcal{G}^{*}$, $\pi_{j}\leq\alpha^{d}_{j},\;j\in\mathcal{D}^{*}$, $\pi_{l}\geq\alpha^{f}_{l},\;l\in\mathcal{L}^{*}$, and $\pi_{m}\geq\alpha^{\xi}_{m},\;m\in\mathcal{M}^{*}$.
\end{theorem}
\begin{proof}
	Theorem \ref{ThmProfit} indicates that profits $(\phi^{g},\phi^{d},\phi^{f},\phi^{\xi})$ are nonnegative. The allocations $(g,d,f,\xi)$ are nonnegative by definition, and strictly positive for cleared stakeholders. Together with the profit identities \eqref{ProfitIdentities} this implies that $\pi_{i}-\alpha^{g}_{i}\geq0,\;i\in\mathcal{G}^{*}$, $\alpha^{d}_{j}-\pi_{j}\geq0,\;j\in\mathcal{D}^{*}$, $\pi_{l}-\alpha^{f}_{l}\geq0,\;l\in\mathcal{L}^{*}$, $\pi_{m}-\alpha^{\xi}_{m}\geq0,\;m\in\mathcal{M}^{*}$.
\end{proof}

Theorems \ref{ThmProfit}, \ref{ThmCompetition},\ref{ThmPrices}, and \ref{ThmLowerBounds} were established for SC-SS \citep{Sampat2019waste}; here, we prove that these properties hold in a dynamical setting. Notably, this generalization is quite straightforward by using the concept of space-time transport flows. 

In addition to the bounds in Theorem \ref{ThmLowerBounds} (emerging from the Lagrangian dual),  we are able to establish the following price bounding behavior from the dual program.

\begin{theorem}\label{ThmUpperBounds}
	The clearing prices $(\pi_{i},\pi_{j},\pi_{l},\pi_{m})$ satisfy the bounds $\pi_{i} - \overline{\lambda}_{i} \leq  \alpha^{g}_{i},\; i\in\mathcal{G}$, $\pi_{j} + \overline{\lambda}_{j} \geq \alpha^{d}_{j},\;j\in\mathcal{D}$, $\pi_{l} - \overline{\lambda}_{l} \leq \alpha^{f}_{l},\;l\in\mathcal{L}$, and $\pi_{m} - \overline{\lambda}_{m} \leq \alpha^{\xi}_{m},\;m\in\mathcal{M}$.
\end{theorem}
\begin{proof}
	The result follows directly from the dual representation of SC-ST \eqref{DualSub}.
\end{proof}

The bounds defined in Theorem \ref{ThmUpperBounds} are upper bounds on player prices (lower for consumers); this means that we now have lower \eqref{PriceLowerBounds} and upper \eqref{PriceUpperBounds} bounds on prices (\textit{vice versa} for consumers). These results are key in understanding space-time dynamic behavior of prices. 

\begin{subequations}\label{PriceLowerBounds}
	\begin{equation}\label{PLBi}
	\begin{aligned}
	(\pi_{i} -  \alpha^{g}_{i})g_{i} \geq 0,\; i\in\mathcal{G}
	\end{aligned}
	\end{equation}
	\begin{equation}
	\begin{aligned}
	(\alpha^{d}_{j} - \pi_{j})d_{j} \geq 0 ,\; j\in\mathcal{D}
	\end{aligned}
	\end{equation}
	\begin{equation}
	\begin{aligned}
	(\pi_{l} -  \alpha^{f}_{l})f_{l} \geq 0,\; l\in\mathcal{L}
	\end{aligned}
	\end{equation}
	\begin{equation}
	\begin{aligned}
	(\pi_{m} - \alpha^{\xi}_{m})\xi_{m} \geq 0,\; m\in\mathcal{M}
	\end{aligned}
	\end{equation}
\end{subequations}

\begin{subequations}\label{PriceUpperBounds}
	\begin{equation}
	\begin{aligned}
	\pi_{i} \leq \alpha^{g}_{i} + \overline{\lambda}_{i},\; i\in\mathcal{G}
	\end{aligned}
	\end{equation}
	\begin{equation}
	\begin{aligned}
	\pi_{j} \geq \alpha^{d}_{j} - \overline{\lambda}_{j},\; j\in\mathcal{D}
	\end{aligned}
	\end{equation}
	\begin{equation}
	\begin{aligned}
	\pi_{l} \leq \alpha^{f}_{l} + \overline{\lambda}_{l},\; l\in\mathcal{L}
	\end{aligned}
	\end{equation}
	\begin{equation}
	\begin{aligned}
	\pi_{m} \leq \alpha^{\xi}_{m} + \overline{\lambda}_{m},\; m\in\mathcal{M}
	\end{aligned}
	\end{equation}
\end{subequations}

We observe that the lower bounds \eqref{PriceLowerBounds} are enforced subject to the corresponding player receiving a positive allocation. Using the supplier class to illustrate, we have $g_{i}>0\implies \pi_{i}\geq\alpha^{g}_{i}$ (from \eqref{PLBi}). Strong duality provides that the $\overline{\lambda}$ dual variables in \eqref{PriceUpperBounds} are positive only if the corresponding stakeholder is allocated its entire capacity, mathematically: $g_{i}=\bar{g}_{i}\impliedby\overline{\lambda}_{i}>0$ for suppliers, with similar logic for the other classes. The $\overline{\lambda}$ are zero otherwise, e.g., $g_{i}<\bar{g}_{i}\implies\overline{\lambda}_{i}=0$. So a supplier with an allocation $0<g_{i}<\bar{g}_{i}$ experiences a market price equal to its bid due to the bounds $\alpha^{g}_{i}\leq\pi_{i}\leq\alpha^{g}_{i}$. The interplay of strong duality and market prices has important implications on how profits are allocated in coordinated markets. Substituting the price bounding results into the corresponding profit definitions from \eqref{ProfitIdentities} we observe $\overline{\lambda}_{i}>0\implies\phi^{g}_{i}\leq\overline{\lambda}_{i}\bar{g}_{i}$ and $\overline{\lambda}_{i}=0\implies\phi^{g}_{i}=0$. Following similar logic, we obtain bounds for the other classes $\overline{\lambda}_{j}>0\implies\phi^{d}_{j}\leq\overline{\lambda}_{j}\bar{d}_{j}$, $\overline{\lambda}_{l}>0\implies\phi^{f}_{l}\leq\overline{\lambda}_{l}\bar{f}_{l}$, and $\overline{\lambda}_{m}>0\implies\phi^{\xi}_{m}\leq\overline{\lambda}_{m}\bar{\xi}_{m}$.

\begin{theorem}\label{ThmProfitBounds}
	A solution of SC-ST delivers clearing prices $\pi$ and allocations $(g,d,f,\xi)$ such that a player can be allocated a positive profit only if it is allocated its entire capacity.
\end{theorem}
\begin{proof}
	We have that the market with players $(\mathcal{G},\mathcal{D},\mathcal{L},\mathcal{M})$ and the set of cleared players $(\mathcal{G}^{*},\mathcal{D}^{*},\mathcal{L}^{*},\mathcal{M}^{*})$ (i.e., $\mathcal{G}^{*}:=\{i\in\mathcal{G}|g_{i}>0\}$). Define the sets $(\mathcal{G}^{\bullet},\mathcal{D}^{\bullet},\mathcal{L}^{\bullet},\mathcal{M}^{\bullet})$ and $(\mathcal{G}^{\circ},\mathcal{D}^{\circ},\mathcal{L}^{\circ},\mathcal{M}^{\circ})$ where $\mathcal{G}^{\bullet}:=\{i\in\mathcal{G}|g_{i}=\bar{g}_{i}\}$ and $\mathcal{G}^{\circ}:=\{i\in\mathcal{G}|0<g_{i}<\bar{g}_{i}\}$, so we have $\mathcal{G}^{*}\setminus\mathcal{G}^{\bullet}=\mathcal{G}^{\circ}$ (with the same for the other classes). Further, define the sets of dry players $(\mathcal{G}^{\odot},\mathcal{D}^{\odot},\mathcal{L}^{\odot},\mathcal{M}^{\odot})$ where $\mathcal{G}^{\odot}:=\{i\in\mathcal{G}|g_{i}=0\}$, having $\mathcal{G}\setminus\mathcal{G}^{*}=\mathcal{G}^{\odot}$. We can express the price bounds from \eqref{PriceLowerBounds} as
	\begin{equation*}
		\begin{aligned}
		\pi_{i}
		\begin{cases}
		\geq \alpha^{g}_{i},& i\in\mathcal{G}^{*}\\
		\in\mathbb{R},& i\in\mathcal{G}^{\odot}
		\end{cases}
		\end{aligned}
	\end{equation*}
	\begin{equation*}
		\begin{aligned}
		\pi_{j}
		\begin{cases}
		\leq \alpha^{d}_{j},& j\in\mathcal{D}^{*}\\
		\in\mathbb{R},& j\in\mathcal{D}^{\odot}
		\end{cases}
		\end{aligned}
	\end{equation*}
	\begin{equation*}
		\begin{aligned}
		\pi_{l}
		\begin{cases}
		\geq \alpha^{f}_{l},& l\in\mathcal{L}^{*}\\
		\in\mathbb{R},& l\in\mathcal{L}^{\odot}
		\end{cases}
		\end{aligned}
	\end{equation*}
		\begin{equation*}
		\begin{aligned}
		\pi_{m}
		\begin{cases}
		\geq \alpha^{\xi}_{m},& m\in\mathcal{M}^{*}\\
		\in\mathbb{R},& m\in\mathcal{M}^{\odot}
		\end{cases}
		\end{aligned}
	\end{equation*}
	
	and from \eqref{PriceUpperBounds} as
	\begin{equation*}
	\begin{aligned}
	\pi_{i}
	\begin{cases}
	\leq \alpha^{g}_{i} + \overline{\lambda}_{i},& i\in\mathcal{G}^{\bullet}\\
	\leq \alpha^{g}_{i},& i\in\mathcal{G}^{\odot}\cup\mathcal{G}^{\circ}
	\end{cases}
	\end{aligned}
	\end{equation*}
	\begin{equation*}
	\begin{aligned}
	\pi_{j}
	\begin{cases}
	\geq \alpha^{d}_{j} + \overline{\lambda}_{j},& j\in\mathcal{D}^{\bullet}\\
	\geq \alpha^{d}_{j},& j\in\mathcal{D}^{\odot}\cup\mathcal{D}^{\circ}
	\end{cases}
	\end{aligned}
	\end{equation*}
	\begin{equation*}
	\begin{aligned}
	\pi_{l}
	\begin{cases}
	\leq \alpha^{f}_{l} + \overline{\lambda}_{l},& l\in\mathcal{L}^{\bullet}\\
	\leq \alpha^{f}_{l},& l\in\mathcal{L}^{\odot}\cup\mathcal{L}^{\circ}
	\end{cases}
	\end{aligned}
	\end{equation*}
	\begin{equation*}
	\begin{aligned}
	\pi_{m}
	\begin{cases}
	\leq \alpha^{\xi}_{m} + \overline{\lambda}_{m},& m\in\mathcal{M}^{\bullet}\\
	\leq \alpha^{\xi}_{m},& m\in\mathcal{M}^{\odot}\cup\mathcal{M}^{\circ}
	\end{cases}
	\end{aligned}
	\end{equation*}
	
	Substitution of the price bounds into the profit definitions results in the profit bounds:
	\begin{equation*}
	\begin{aligned}
	\phi_{i}
	\begin{cases}
	\leq \overline{\lambda}_{i}\bar{g}_{i},& i\in\mathcal{G}^{\bullet}\\
	\leq 0,& i\in\mathcal{G}^{\odot}\cup\mathcal{G}^{\circ}\\
	\geq 0,& i\in\mathcal{G}
	\end{cases}
	\end{aligned}
	\end{equation*}
	\begin{equation*}
	\begin{aligned}
	\phi_{j}
	\begin{cases}
	\leq \overline{\lambda}_{j}\bar{d}_{j},& j\in\mathcal{D}^{\bullet}\\
	\leq 0,& j\in\mathcal{D}^{\odot}\cup\mathcal{D}^{\circ}\\
	\geq 0,& j\in\mathcal{D}
	\end{cases}
	\end{aligned}
	\end{equation*}
	\begin{equation*}
	\begin{aligned}
	\phi_{l}
	\begin{cases}
	\leq \overline{\lambda}_{l}\bar{f}_{l},& l\in\mathcal{L}^{\bullet}\\
	\leq 0,& l\in\mathcal{L}^{\odot}\cup\mathcal{L}^{\circ}\\
	\geq 0,& l\in\mathcal{L}
	\end{cases}
	\end{aligned}
	\end{equation*}
	\begin{equation*}
	\begin{aligned}
	\phi_{m}
	\begin{cases}
	\leq \overline{\lambda}_{m}\bar{\xi}_{m},& m\in\mathcal{M}^{\bullet}\\
	\leq 0,& m\in\mathcal{M}^{\odot}\cup\mathcal{M}^{\circ}\\
	\geq 0,& m\in\mathcal{M}
	\end{cases}
	\end{aligned}
	\end{equation*}
	The profit bounds identified exhaust all possible player outcomes. Only players in the sets $(\mathcal{G}^{\bullet},\mathcal{D}^{\bullet},\mathcal{L}^{\bullet},\mathcal{M}^{\bullet})$ have a positive profit bound, completing the proof.
\end{proof}

Theorem \ref{ThmProfitBounds} informs the distribution of profit to market stakeholders. Though the clearing problem can posses degeneracy (solution multiplicity), it is limited to a specific subset of stakeholders in any particular outcome. Given our analysis, it is necessary to confirm that there will be at least one stakeholder that is cleared in a non-dry market with a positive profit bound whom can be allocated profit by the ISO.

\begin{theorem}\label{ThmAtLeastOne}
	A solution of SC-ST delivers clearing prices $\pi$ and allocations $(g,d,f,\xi)$ such that at least one market player has a strictly positive bound on its profit allocation in a non-dry market.
\end{theorem}
\begin{proof}
	Proof is by contradiction and relies on the extreme point properties of linear programming solutions. Assume we have a market with an optimal set of clear transactions $(\mathcal{G}^{*},\mathcal{D}^{*},\mathcal{L}^{*},\mathcal{M}^{*})$ all satisfying $(g^{*}_{i},d^{*}_{j},f^{*}_{l},\xi^{*}_{m})<(\bar{g}_{i},\bar{d}_{j},\bar{f}_{l},\bar{\xi}_{m})$, meaning $(\mathcal{G}^{\bullet},\mathcal{D}^{\bullet},\mathcal{L}^{\bullet},\mathcal{M}^{\bullet})=\varnothing$. The objective value is $z^{*}=\sum_{j \in \mathcal{D}}{\alpha^{d}_{j}}d^{*}_{j} -
	\sum_{i \in \mathcal{G}}{\alpha^{g}_{i}}g^{*}_{i} -
	\sum_{l\in\mathcal{L}}{\alpha^{f}_{l}f^{*}_{l}} - 
	\sum_{m\in\mathcal{M}}{\alpha^{\xi}_{m}\xi^{*}_{m}}$. Since none of the allocation bounds $(\bar{g}_{i},\bar{d}_{j},\bar{f}_{l},\bar{\xi}_{m})$ are active, the ISO can increase its allocations by amounts $(\Delta g_{i},\Delta d_{j},\Delta f_{l},\Delta \xi_{m})$ subject to the market clearing constraints until one or more of the allocations reaches its bound. The solution $(g^{*}_{i}+\Delta g_{i},d^{*}_{j}+\Delta d_{j}f^{*}_{l}+\Delta f_{l},\xi^{*}_{m}+\Delta \xi_{m})$ produces an objective value $z^{*}+\Delta z$, and if $z^{*}$ is optimal, then $z^{*}+\Delta z > z^{*}$, which is a contradiction. Therefore, in any non-dry market, there will be at least one stakeholder in the sets $(\mathcal{G}^{\bullet},\mathcal{D}^{\bullet},\mathcal{L}^{\bullet},\mathcal{M}^{\bullet})$ having a positive upper bounds on its allocatable profits.
\end{proof}

We next examine space-time price dynamics; here, we are interested in the concept of space-time price volatility (variability) and how we can design markets to mitigate price volatility for market stakeholders. Price volatility may manifest in both the spatial and temporal dimensions, with spatial volatility manifesting as price variation between spatial nodes $n\in\mathcal{N}$ and temporal volatility between temporal nodes $t\in\mathcal{T}$. Transport providers play a key role in driving and controlling space-time volatility, as we see next. 

\begin{theorem}\label{ThmTransportPrices}
	Transport providers can drive space-time price volatility to zero.
\end{theorem}
\begin{proof}
	The transport price in \eqref{TransportPrice} is subject to bounds defined by the dual program \eqref{DualSub} and the Lagrangian dual \eqref{LDual}, from we derive the bounds
	\begin{equation*}
	\begin{aligned}
	\alpha^{f}_{l} \leq \pi_{l} \leq \alpha^{f}_{l} + \overline{\lambda}_{l}, \quad l\in\mathcal{L}\;|\;f_{l}>0
	\end{aligned}
	\end{equation*}
	for transport providers with positive allocations. If a transport provider does not receive a profit, either because its allocation $f_{l}$ is less than its transport capacity $\bar{f}_{l}$ or due to degeneracy, then these bounds become:
	\begin{equation*}
	\begin{aligned}
	\alpha^{f}_{l} \leq \pi_{l} \leq \alpha^{f}_{l}, \quad l\in\mathcal{L}\;|\;f_{l}>0
	\end{aligned}
	\end{equation*}
	implying $\pi_{n,t,p}=\pi_{n',t',p}$ for nodes $(n,n')\in\mathcal{N}$ and $(t,t')\in\mathcal{T}$ when $\alpha^{f}_{l}=0$.
\end{proof}

Transport bids $\alpha^{f}_{l}$ lower than local supply bids (intuitively) allow access to product sources at other nodes or time points with favorable prices. Perhaps the most important interpretation of these properties is that the spatial and temporal dimensions of the clearing model are fundamentally the same. This indicates, for instance, that storage systems and geographical transport systems play a key role in determining price dynamics. A couple of simple example problems that illustrate these theoretical properties are included as supplementary material. The first demonstrates theoretical properties related to profit allocations and prices, and how those theorems may be interpreted in practice. The second focuses on the behavior of temporal transportation. We also comment on how theoretical properties are observed in the context of our case study. 

\section{Case Study}

We illustrate our theoretical developments by considering the waste-to-energy case study described by \citet{Hu2018}. Here, we analyze the potential creation of a coordinated livestock waste (manure) market (a bioeconomy) in the State of Wisconsin that can be used to generate valuable energy products (e.g., biogas, electricity). This problem can be cast as a supply chain problem and can be interpreted as a coordinated market in which suppliers of waste (dairy farmers) seek to satisfy demands of valuable products derived from waste (e.g., biogas and electricity). Specifically, in this market we would use waste processing technologies (comprised of anaerobic digestion and power generators) to produce biogas and electricity from manure. Moreover, in this market we would have geographical transport of waste and temporal transport via waste storage. This market can also help mitigate myriad environmental issues associated with manure management; specifically, the practice of spreading manure on crop fields (as a fertilizer) leads to uncontrolled degradation of organic matter contained in manure and leads to methane and pathogen emissions \citep{Sharara2017}. Unfortunately, manure processing technologies are expensive and it is thus desirable to ensure that biogas and electricity generate sufficient revenue.  

Here, we reexamine this problem, using the same data as \citet{Hu2018} (including the disposition of farms, biogas technologies, storage capacities, and technology costs) and identify new revenue opportunities for manure processing by exploiting the dynamics of electricity prices.  We aim to show that manure processing systems can respond to dynamic price incentives in the same way that modern natural gas power plants can (because they are able to rapidly commit energy to the grid). We estimate dynamic electricity demand across the State of Wisconsin using the demand curve from \citet{zavala2010dynamic} scaled to the Wisconsin statewide annual electricity consumption rate of 68.8 TWh per year \citep{DOEProfile}, illustrated in Figure \ref{FigDemandCurve}. We also fit the generation curves of \citet{zavala2010dynamic} (representing three distinct conventional grid power suppliers) to reflect reasonable values of Wisconsin off and on-peak electricity prices of 0.05 and 0.18 USD/kWh, respectively (based on 2019 real-time market price data from the \citet{MISO2021}) (illustrated in Figure \ref{FigSupplyCurves}); i.e., the curves are fit to yield these prices at the corresponding on-peak and off-peak output levels. The three supply curves are fit using simple quadratic equations of the form $y = \beta\, x^{2}$ where the parameters $g$ are [1.66\e{-5},8.31\e{-6},4.15\e{-5}]. We retain the three separate curves from \citet{zavala2010dynamic} as a representation of different generator types, with varying costs, but discretize them to intervals of 100 MW of output. Each of these output levels is treated as a separate supplier; this illustrates how the proposed modeling framework can capture nonlinear bidding costs.  We model one week of dynamic behavior (168 hour long periods) using this data.

\begin{figure}[!htb]
	\center{\includegraphics[width=100mm]{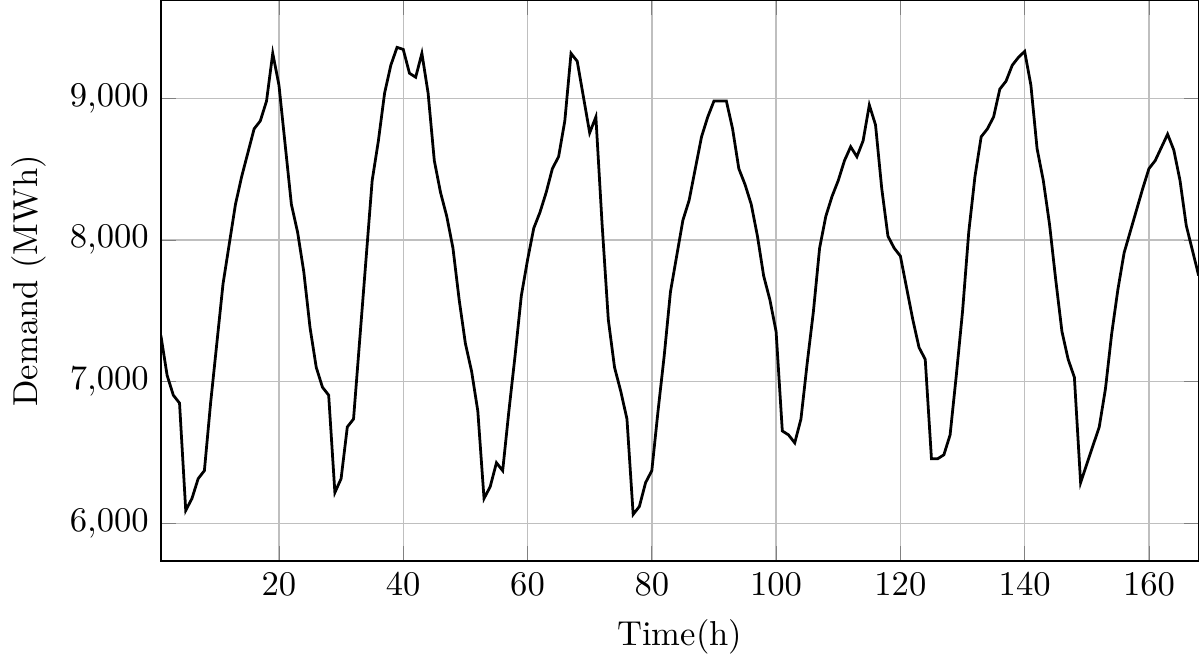}}
	\caption{Electricity demand curve scaled to Wisconsin annual consumption levels. The curve represents one week (168 hours) of demand.}
	\label{FigDemandCurve}
\end{figure}

\begin{figure}[!htb]
	\center{\includegraphics[width=100mm]{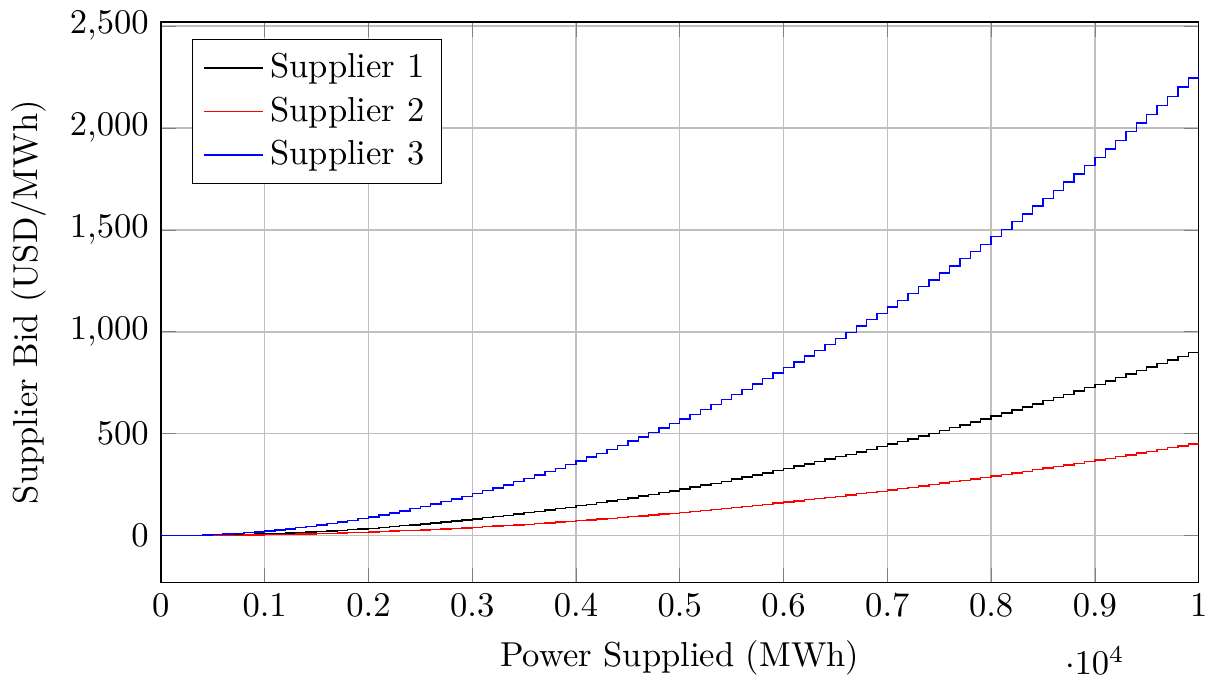}}
	\caption{Electricity supply available from grid producers. Three producers are modeled nonlinear generation curves, representing different conventional electricity producers.}
	\label{FigSupplyCurves}
\end{figure}

The dairy infrastructure comprises 245 farms (concentrated animal feeding operations, or CAFOs) of which 120 are equipped with waste processing and biogas systems. Our problem setup corresponds to the biogas gas study of \citet{Hu2018}, and that paper provides the complete data for farms and technology specifications. We assume a single collection point for the electricity generated from biogas, centered on the City of Madison; since the transportation cost of electricity is contextually small (estimated at 7.5\e{-6} USD/MWh$\cdot$km based on average transmission losses of 5\% per 1000 km reported by \citet{Vaillancourt2014}) this assumption should not have a significant impact on the qualitative nature of our results. The disposition of CAFOs and processing systems are visualized in Figure \ref{FigNodeMap}, which showcases the average hourly production rates of dairy waste throughout the state.

\begin{figure}[!htb]
	\center{\includegraphics[width=100mm]{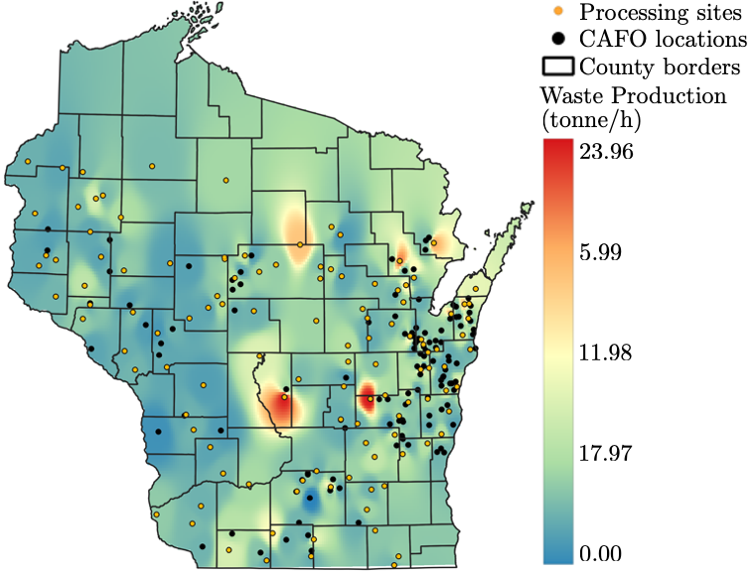}}
	\caption{Hourly dairy waste production rate at CAFOs across Wisconsin (colorbar). Biogas infrastructure and farm distribution are shown by filled circles. CAFO locations are indicated by black circles, CAFOs with storage and biogas processing equipment have a gold center.}
	\label{FigNodeMap}
\end{figure}

Our spatiotemporal model has 168 time periods (each one hour long) with 246 spatial nodes (245 CAFOs and the electricity collection hub). Arcs are constructed connecting all CAFOs to each CAFO with processing facilities on site (i.e., the 125 CAFOs without processing systems are each connected to the 120 equipped with them) and these 120 CAFOs have arcs connecting them to the collection hub at each time point. We assume that these transportation flows occur within a single time period. Temporal flows are interpreted as waste storage; these are implemented as arcs connecting CAFOs over the temporal dimension. To mitigate model size, arcs are constructed across single time points only (i.e., from hour 1 to hour 2, and from 2 to 3, but not from 1 to 3 directly) to model the dynamics of waste storage. This setup also illustrates how we can use our formulation to model various physical phenomena.

In addition to our base model, we include three variations that help illustrate the theoretical properties associated with the model itself, and that provide some useful insights into the base case solution. The cases are the base case, a case with no waste storage (as though electricity can only be produced in the same period that waste is supplied) a case with unlimited storage (storage being both free of cost and effectively unlimited in capacity) and finally a case study in which the base case waste supply is tripled.

We implemented and solve all problems in the Julia programming language (version 1.5.3) \citep{Bezanson2017} and the JuMP modeling language (version 25.1) \citep{Dunning2017} with the Gurboi solver (version 9.1.1) \citep{gurobi}. We characterize the model solution using market prices of electricity and waste, the amount of waste processed to produce biogas, and illustrate some of the dynamics with specific instances of variable values. Overall, our model suggests that all of the dairy waste produced at CAFOs in Wisconsin can be profitably processed to produce biogas and electricity by taking advantage of price fluctuations. Due to the low yield of electricity from waste, the electricity contributed this way is small on the state scale (WI consumes about 17,100 MWh each week, with our statewide biogas network averaging about a 102 MW power rating) the base case total energy supply, including both the conventional grid and contributions from biogas, is illustrated in Figure \ref{FigTotalMWhSupply}.

\subsection{Base Results}

\begin{figure}[!htb]
	\center{\includegraphics[width=100mm]{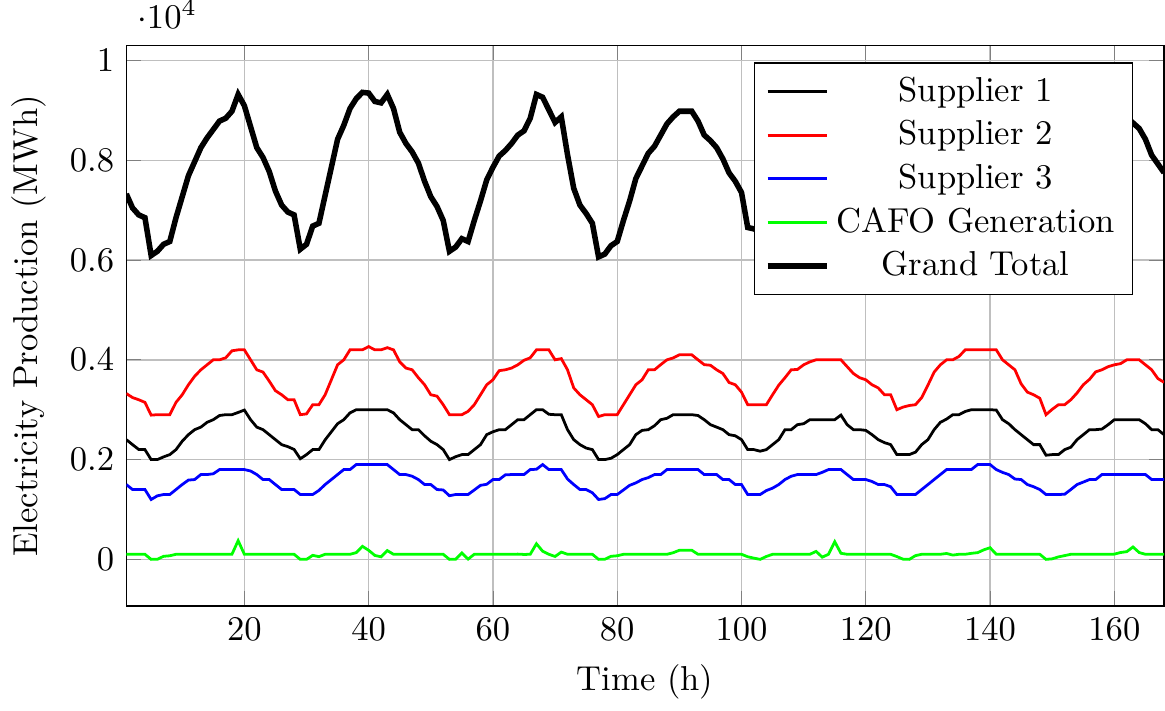}}
	\caption{Total energy supply to the Wisconsin grid over a one week period. Contributions from conventional suppliers and the aggregate production at CAFOs are shown. The grand total replicates the demand curve.}
	\label{FigTotalMWhSupply}
\end{figure}

Base electricity generation rates at various time points (chosen to illustrate temporal variation) are shown in Figure \ref{FigBioenergySupply}. Electricity generation is concentrated at two major CAFOs, with marginal production distributed throughout the state (the figures are presented in log scale to highlight smaller contributions). The disposition of electricity generation throughout the state is influenced by waste availability, transport costs, and local technology availability. We replicate technology placement from \citet{Hu2018}, but observe that a different distribution of technologies may be optimal in a dynamic setting. Nonetheless, we observe temporal variation, with electricity generation at CAFOs following peak demand times and none during off-peak hours when prices are less favorable. The pattern is most prominent between hour 5 and hour 19, which capture the low and high extents of electricity generation rates throughout the state.

Base case hourly waste storage levels are shown in Figure \ref{FigStorageMap} and demonstrate that coupling to the state electrical grid creates an incentive to store and use waste for electricity production. The storage dynamics are cyclic; we observe that waste storage units begin empty, gradually fill ($t$=5) to capacity throughout the state ($t$=11) and then empty ($t$=19) as biogas is used to produce electricity. The cycle repeats over subsequent days. Storage tanks remain full for roughly fourteen hours each day in this case, and empty for ten. This pattern is indicative of biogas accumulation and processing patterns.

\begin{figure}[!htb]
	\center{\includegraphics[width=\textwidth]{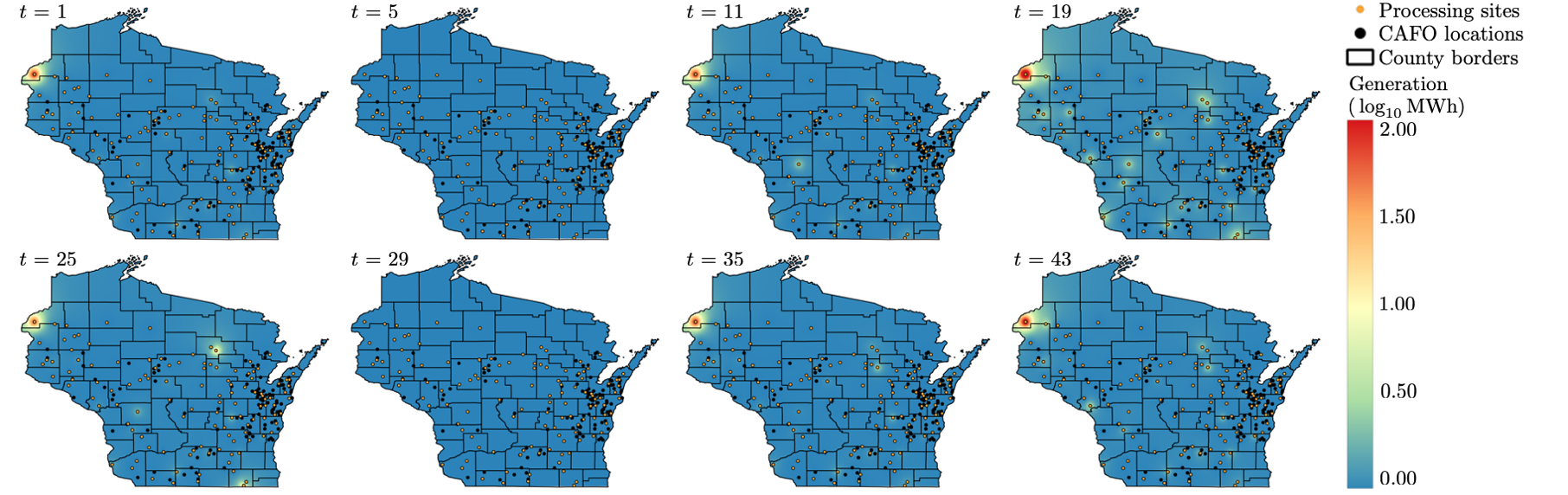}}
	\caption{Electricity produced from biogas (MWh). Peak hourly generation exceeds 300 MWh, primarily focused on two major CAFOs. The graph on the right present log(MWh) values to highlight smaller levels of generation on the order of 2 - 5 MWh.}
	\label{FigBioenergySupply}
\end{figure}

\begin{figure}[!htb]
	\center{\includegraphics[width=150mm]{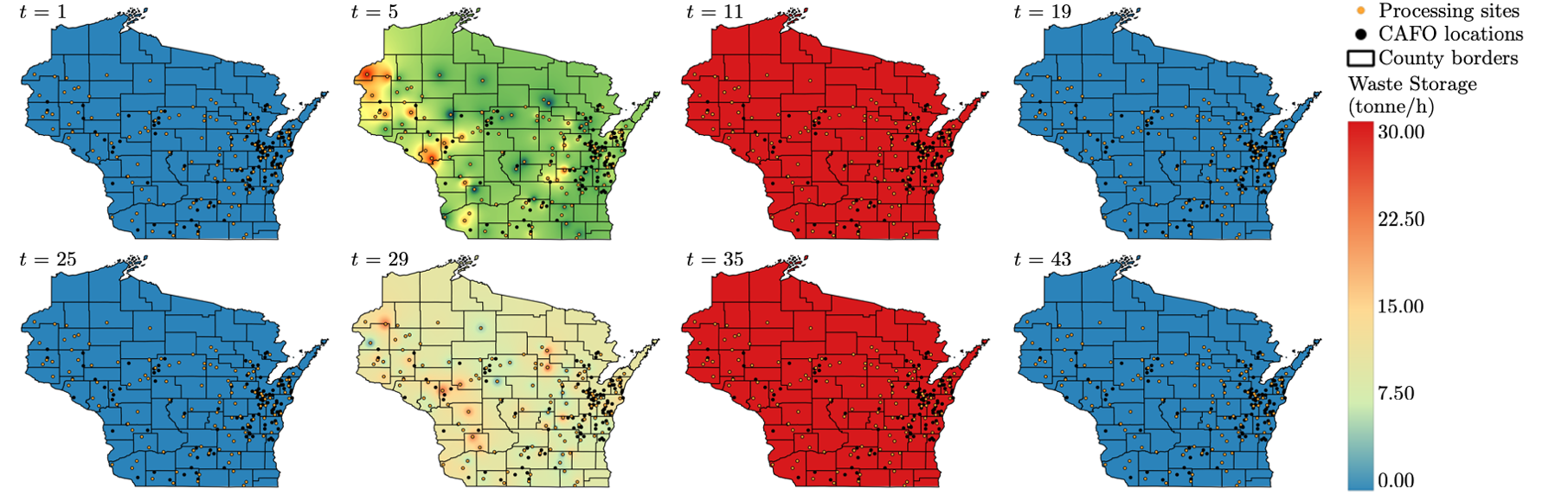}}
	\caption{Distribution of hourly waste storage at select time points. Storage units fill to capacity and empty cyclically over the horizon. Coupling with the electricity grid creates temporally varying incentives for waste storage.}
	\label{FigStorageMap}
\end{figure}

\subsection{Comparative Case Studies}

We now present results from all case studies for comparison. To begin, we present the revenue streams from each of our four case studies in Table \ref{TableCaseRevenues}. Here, we present the total revenue spent (negative) or collected (positive) from the market over all stakeholders of a given class. We include the grand total values (all zero) because these demonstrate Theorem \ref{ThmPrices} (revenue adequacy), which guarantees that revenue balances in the market. Importantly, this also demonstrates Theorem \ref{ThmCompetition} (competitiveness) because we observe that all revenues paid into the market are collected by other market players; i.e., ISO coordination does not incur costs, and does not introduce inefficiencies, it simply accelerates the clearing procedure. Specifically, consumer revenue represents statewide spending on electricity. Supplier revenue includes payments both to conventional electricity producers and to CAFOs. Payments to transportation and technology providers (these also being associated with CAFOs) are listed separately.

\begin{table}
	\caption{Revenue streams for supply chain consumers, suppliers, transport providers (separated by spatial and temporal dimensions) technologies, and grand total values. These demonstrate revenue adequacy in the market clearing procedure.}
	\begin{tabular}{|l|cccc|}
		\hline
		Revenue stream & Base case & No storage & Unlimited storage & Triple waste \\\hline
		Consumer total & -169,019,002 & -169,554,946 & -167,596,601 & -160,798,068 \\
		Supplier total & 168,833,665 & 169,508,659 & 167,550,314 & 160,515,484 \\
		Transport (temporal) total & 137,214 & 0 & 0 & 143,522 \\
		Transport (spatial) total & 47,807 & 45,971 & 45,971 & 138,114 \\
		Technologies total & 316 & 316 & 316 & 948\\
		Grand Total & 0 & 0 & 0 & 0 \\\hline
	\end{tabular}
\label{TableCaseRevenues}
\end{table}

The 245 CAFOs (with 120 technology-enabled sites) together contribute to the base case generation profile in Figure \ref{FigBioenergySupply}, which impacts the market price of electricity as shown in Figure \ref{FigElectricityPriceCases}, which shows the price of electricity over the time entire horizon at the collection node. Electricity prices resulting in our base case setting (solid black line) are lower than those emerging in the no storage case (i.e., when electricity is available from conventional grid suppliers \textit{only}) (dotted line). Interestingly, electricity prices in the no storage case follow the dynamics of the demand curve (Figure \ref{FigDemandCurve}) replicating its shape. Notably, electricity prices in the base case are lower at peak times, illustrating how the coordination system takes advantage of peak electricity prices by concentrating biogas generation during peak hours, electricity prices are lower than they would be with only the conventional electricity producers. The base and no storage cases specifically demonstrate Theorem \ref{ThmSocialWelfare}; we observe that waste storage allows CAFOs to take advantage of real-time electricity price dynamics to make money and reduce electricity prices for consumers. The unlimited storage and triple waste cases also result in lower electricity prices; notably tripling the amount of waste available reduces prices at all time points, while removing storage limits reduces peak prices only. This result is important; increasing the amount of waste allows CAFOs to profit in the electricity market, but it is storage that allows CAFOs to take advantage of peak hour pricing.

\begin{figure}[!htb]
	\center{\includegraphics[width=100mm]{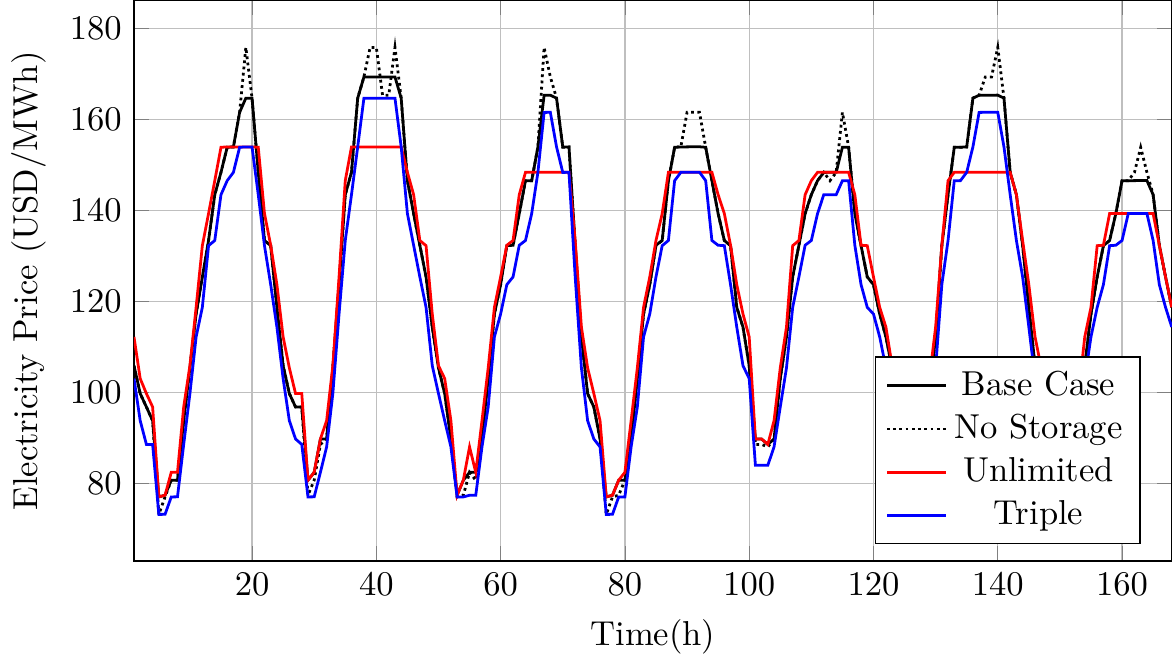}}
	\caption{Electricity prices resulting under (solid black) base case conditions, (dotted black) no storage, (red) with unlimited storage, and (blue) with triple the amount of waste available. Biogas systems are able to take advantage of demand peaks to provide electricity at lower price during peak hours, which reduces peak pricing, and creates value for dairy farmers.}
	\label{FigElectricityPriceCases}
\end{figure}

The greatest decrease in base case price observed is 1.26 cents per kWh (12.6 USD/MWh) representing a significant peak savings, noting that our peak price is around 18 cents per kWh. Over the course of the 168 hour horizon, the average difference between the electricity price with and without the biogas contribution is only 0.34 cents per kWh, demonstrating that the effects are concentrated at the peaks. We suggest that biogas represents an opportunity to reduce electricity prices during peak hours, creating value for consumers and for dairy farmers who can realize the inherent value in waste. Specifically, we observe that the value of waste is linked to storage. Total waste storage state-wide is illustrated in Figure \ref{FigStorageProfiles}, from which we observe that storage is used to its full extent in both the base and triple cases, and in the unlimited storage case we observe significantly greater usage of waste storage, suggesting that the ability to take advantage of peak prices is limited by existing storage capacity. This is reflected in Figure \ref{FigElectricityPriceCases}.

\begin{figure}[!htb]
	\center{\includegraphics[width=100mm]{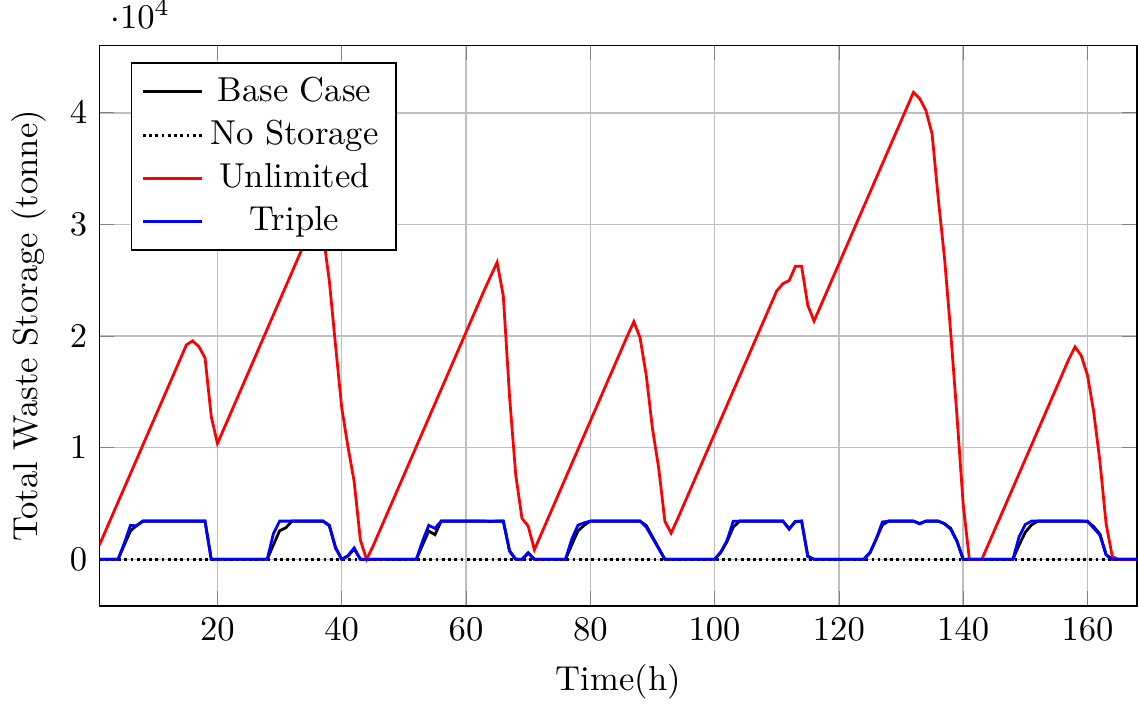}}
	\caption{State-wide waste storage profiles in each of our four case studies. Waste storage is coupled to waste and electricity prices.}
	\label{FigStorageProfiles}
\end{figure}

In Figure \ref{FigWastePriceCases} (which shows the price of waste at a particular CAFO node) we observe that waste prices are driven by dynamics of electricity prices and by waste storage. The peak base case waste price is 13.15 USD/tonne, illustrating how market coordination captures the inherent value of this resource through its potential in the electricity market. In the base case, we observe less waste price variation than in the no storage case, while the triple waste case results in lower waste prices overall than in the base case. The unlimited storage case demonstrates the extreme mitigation of waste price over time, with all the induced electricity price dynamics effectively balanced by storage dynamics resulting in a stable value of waste around 12 USD/tonne. This demonstrates Theorem \ref{ThmTransportPrices}, having eliminated the price volatility in waste prices, and mitigated volatility in electricity prices. There is spatial variation in waste prices as well, though it is on a smaller scale than the temporal variation. Figure \ref{FigWastePriceMaps} illustrates these scales for the base case; temporal price variation (more than 8.00 USD/tonne) dominates spatial variation (approximately 0.30 USD/tonne).

\begin{figure}[!htb]
	\center{\includegraphics[width=100mm]{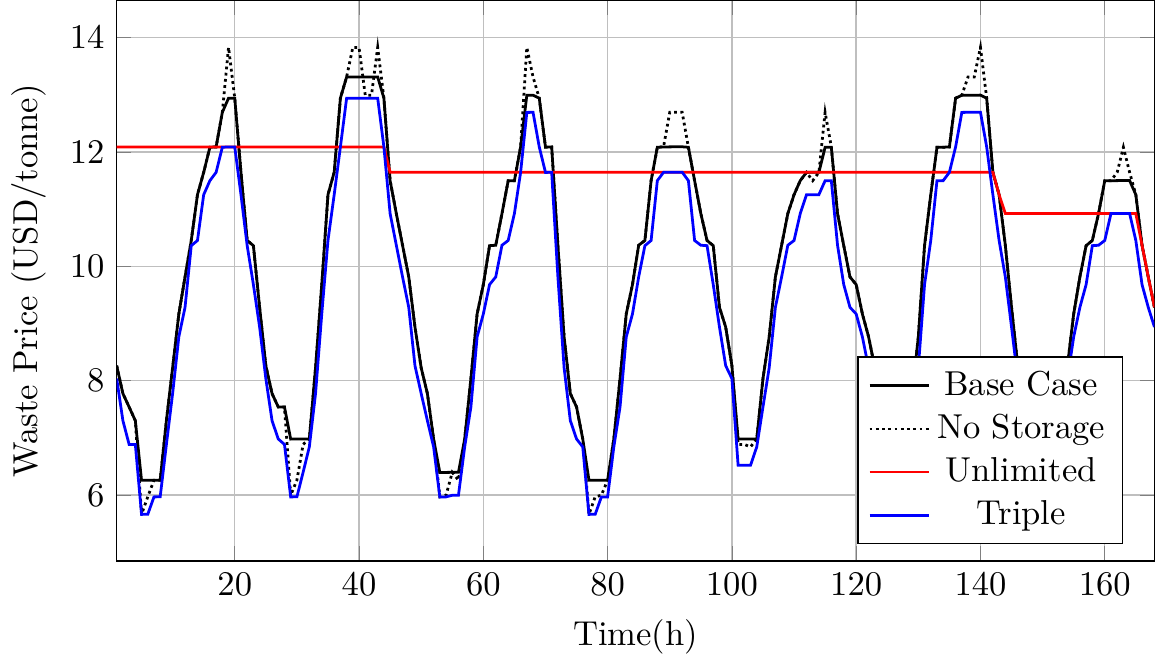}}
	\caption{Waste prices at a particular CAFO are induced the dynamics of electricity prices at the collection node, which translates into a peak value of 13.15 USD per tonne in the base case. With unlimited storage, the price of waste reaches a stable value.}
	\label{FigWastePriceCases}
\end{figure}

\begin{figure}[!htb]
	\center{\includegraphics[width=150mm]{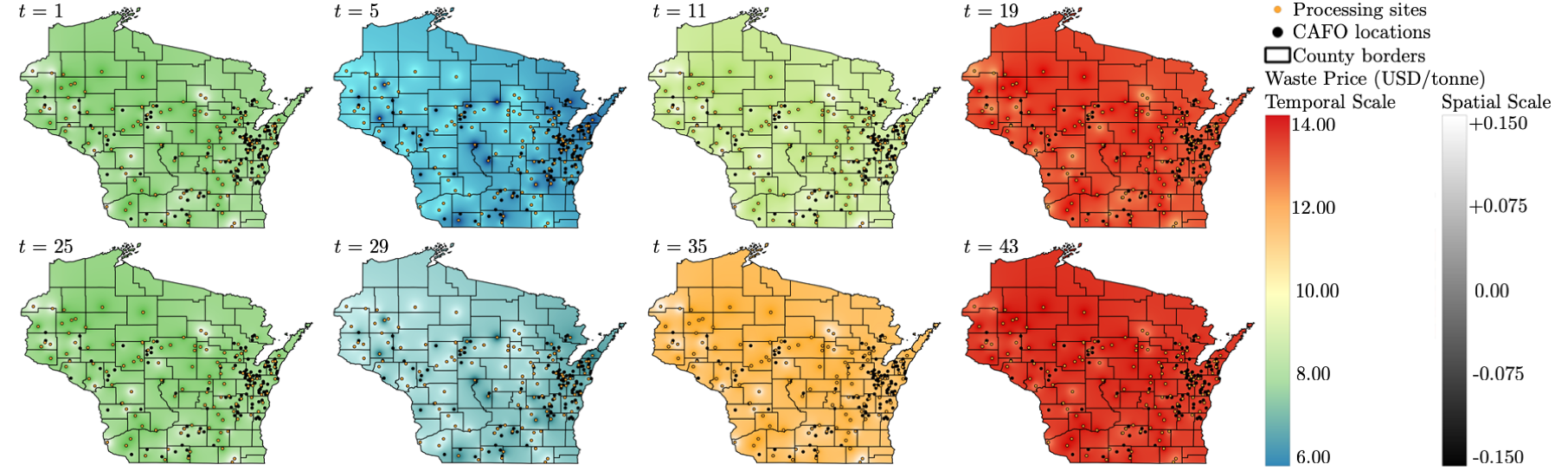}}
	\caption{Base case waste price dynamics show significantly more temporal variation than spatial, with waste value heavily influenced by electricity prices. Spatial variation is influenced by transportation costs. Due to the difference in scale, two separate color bars are used for spatial and temporal price variation. The primary color bar indicates temporal waste prices, while a gray scale overlay captures temporal variation at each time point.}
	\label{FigWastePriceMaps}
\end{figure}

We can interpret price and storage dynamics through electricity generation profiles. In Figure \ref{FigGenerationCases} we illustrate the combined generation of all 245 CAFOs over the 168 hour horizon. Notably, the base case has generation concentrated on the on-peak hours, but due to the storage limitations associated with this case, we also observe generation in off-peak hours. Tripling the amount of waste available does not change the qualitative nature of this result; the curves are similar. In contrast, the no storage case has a constant generation rate of about 100 MWh. In the unlimited storage case, we observe no generation during off-peak hours, with generation concentrated on peak prices. This provides some insight into the limiting nature of storage in this problem, and suggests that the flexibility that storage provides allows farmers to take advantage of the highest prices.

\begin{figure}[!htb]
	\center{\includegraphics[width=100mm]{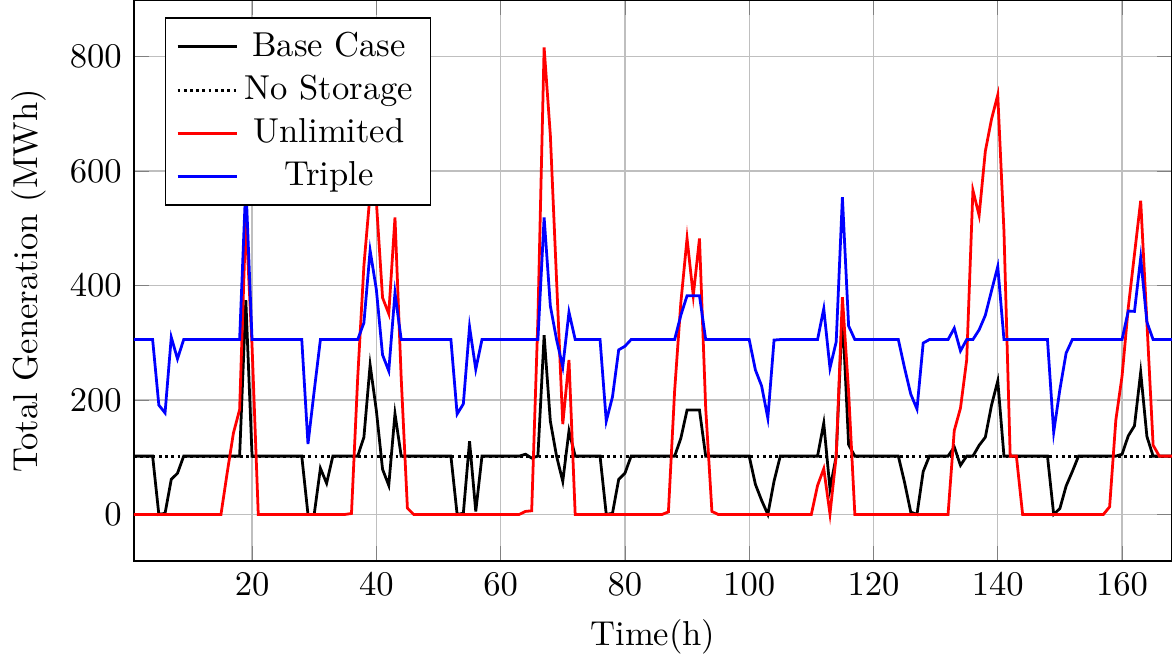}}
	\caption{Combined electricity generation over 245 WI CAFOs. Unlimited storage allows the coordination system to take advantage of real time price peaking, compared with the base case where there is some generation during off-peak hours.}
	\label{FigGenerationCases}
\end{figure}

\subsection{Profit Implications for Dairy Farmers}

In our base result, dairy farms are able to derive (in aggregate) 2,095,953 USD in profit over the course of the 168 hour horizon (extrapolating to an entire year, this could reach 100 million USD). This profit is a result of targeting peak prices in the electricity market through careful planning of waste processing and biogas generation. We have observed how coordination uses these systems, and how limitations in storage influence our results. Averaging our profit value over 245 CAFOs suggests profit on the order of 8,555 USD per CAFO per week, or about 444,855 USD per CAFO per year from participating in the electricity market. Note that we have ignored the disposition of electricity generation in this estimate in favor of an average per-CAFO value.

In our triple-waste case study we observe that all waste is consumed, and peak electricity prices are reduced even further than in the base case. Figure \ref{FigElectricityPriceCases} illustrates this in contrast with the base case and the modified storage cases. From this tripled case, we interpret that there appears to be significant room for growth in the Wisconsin dairy industry. Coupling dairy waste processing to the electrical grid has the potential to generate revenue streams that incentivize profitable waste processing for dairy farmers. The digestate resulting from this process presents a subsequent opportunity, with prospective solutions (like shipping the digestate to nutrient deficient locations in other states as a fertilizer, or processing it further to produce fertilizer products like struvite) not facing the full cost of processing due to the revenue stream created by electricity market participation. The next logical step is coupling digestion to a fertilizer industry that addresses   phosphorus issues (because the electricity grid coupling does nothing to solve the phosphorous issue on its own; it is the revenue stream that is important in this context). Coupling to the electricity grid lowers the cost barriers for market entry in other areas.

It is also important to note that there are other barriers to dairy waste processing that we have not included in our case studies. There are inefficiencies associated with biogas storage (compression and equipment costs) and farms typically use some fraction of the electricity they generate as a part of anaerobic digester operation through combined heat and power (CHP) systems. There is also competition from biogas RINs. These practices reduce the profitability of electricity sales, with the value of the power sold to grid functionally amortized over a greater amount of waste; i.e., the inefficiencies reduce the value of waste. These factors are difficult to address and will be studied in future work; here, the proposed market framework can provide a valuable tool in doing so. 

\section{Conclusions}

We have presented a graph-based dynamic coordinated market framework for multiproduct SC optimization, demonstrating that spatiotemporal transportation induces temporal dynamics in product prices. This unified spatiotemporal framework captures the inherent value of products stemming from interactions between market stakeholders: suppliers, consumers, technology providers, and transportation providers. In particular, our framework captures product transformation, demonstrating that technology prices capture the relative values of inputs and outputs, alongside and in concert with spatiotemporal product transportation. Primal, dual, and Lagrangian formulations are documented, and are used to establish market pricing properties, which provide insight into ISO behavior. Future developments in the market framework will reconcile specific physical phenomena that are desirable in supply chain models, but seem incompatible with the economic formulation.

We illustrate the utility of this model by returning to a previously published problem based on biogas in the Wisconsin dairy industry, and show that by capturing the daily variation of electricity prices, there is a window of opportunity in which generating electricity from dairy biogas is profitable. Moreover, this window is wide enough to accommodate substantially larger amounts of dairy biogas. Our results suggest that dairy farmers could profit (on average) at a rate on the order of 445,000 USD annually, an outcome that motivates techno-economic analysis of the proposed system, accounting for a variety of complicating unmodeled factors, to determine the potential realizable value of this biogas.

\section{Acknowledgments}
We acknowledge support from the U.S. Department of Agriculture (grant 2017-67003-26055) and from the National Science Foundation (under grant 1832208).

\section{References}
\bibliographystyle{model5-names}\biboptions{authoryear}
\small{
\bibliography{SpatiotemporalMarkets.bib}
}

\appendix 
\end{document}